\documentclass[11pt,reqno]{amsart}
\usepackage{amsmath,amsthm, amscd, amssymb, amsfonts,color,mathtools}
\usepackage{hhline}
\usepackage[mathscr]{eucal}
\usepackage{tikz-cd}
\usepackage[active]{srcltx}
\usepackage{enumitem}
\usepackage{charter} 
\usepackage{dsfont}
\usepackage{mathrsfs}
\usepackage{multicol}
\usepackage[all]{xy}

\usepackage[backend=biber,style=alphabetic,sorting=anyt]{biblatex} 
\newtheorem{lemma}{Lemma}[subsection]
\newtheorem{thm}[lemma]{Theorem}

\newtheorem{prop}[lemma]{Proposition}
\newtheorem{remark}[lemma]{Remark}
\newtheorem{example}[lemma]{Example}
\newtheorem{defi}[lemma]{Definition}
\newtheorem{nota}[lemma]{Notation}

\newcommand{\Oc}{\mathcal{O}}
\newcommand{\ord}{\text{ord}}
\newcommand{\Uc}{\mathcal{U}}
\newcommand{\uc}{\mathit{u}}
\newcommand\co{\operatorname{co}}
\newcommand\Ker{\operatorname{Ker}}
\newcommand\Alg{\operatorname{Alg}}
\def\id{\mathop{\rm id} \nolimits}
\def\B{\textbf{B}}
\def\Z{\mathbb{Z}} 
\def\N{\mathbb{N}} 
\def\C{\mathbb{C}}

\def\Q{\mathbb{Q}}

\def\S{\mathcal{S}}

\title[Quantum subgroups of $SL_q(2)$]
{Quantum subgroups of 
$SL_q(2)$ 
at roots of unity of arbitrary order}

\author{G.~A.~Garc\' ia}
\address{
Guangdong-Technion Israel Institute of Technology, No. 241, Daxue Road, 
Shantou, Guangdong Province, China.}
\email{gaston.garcia@gtiit.edu.cn}

\address{
G. A. G. Current Address: CMaLP, Departamento de Matem\'atica, Facultad de Ciencias Exactas,
Universidad Nacional de La Plata. CIC-CONICET. (1900) La Plata, Argentina.}
\email{ggarcia@mate.unlp.edu.ar}

\author{J.~Vallejos}
\address{J.V.: Departamento de Matem\'atica, Universidad Nacional del Sur.
INMABB-CONICET, Bah\'ia Blanca, Argentina.}
\email{josefina.vallejos@uns.edu.ar}

\date{\today}

\date{} 

\begin{document}

\begin{abstract}
We complete the classification of quantum subgroups of $SL_q(2)$ with 
$q$ a root of unity of arbitrary order, that is, Hopf algebra quotients of the 
quantum function algebras 
$\mathcal{O}_{q} (SL_2(\C))$.
\end{abstract}

\maketitle

\section{Introduction}

This paper presents a study of the quantum subgroups of $SL_q(2)$ at roots of unity of arbitrary order, which correspond to Hopf algebra quotients of $\Oc_q(SL_2(\C))$. The problem of classifying this family of quantum subgroups was initiated by E. Müller, who describes the finite dimensional quantum subgroups of $SL_q(n)$ in the case when $q$ is a root of unity of odd order. This was later generalized in \cite{AG} for $G_q$ with $G$ a connected, simply connected, simple complex algebraic group and $q$ a primitive $\ell$-th root of unity, with $\ell$ odd and $3 \nmid \ell$ if $G$ is of type $G_2$. The case $q=-1$ has been studied in \cite{Bichon}, focusing on finite dimensional quantum subgroups; and the case where the parameter is an even root of unity appeared in \cite{Chelsea}, where they studied when the finite subgroups of $SL_q(2)$ act on Artin–Schelter regular algebras $R$ of global dimension 2. Quantum subgroups of other quantum groups have also been studied; these can be seen, for example, in \cite{GJ} and references therein.

Using these results and generalizing the techniques developed in \cite{Mu1}, we extend the classification to all quantum sobgroups of $SL_q(2)$ at an even root of unity. Our goal is a comprehensive list of these subgroups, presented explicitly via generators and relations.

A feature of our classification is that the structure of the subgroups differs significantly depending on the specific parameter. We therefore distinguish two main cases in our study:

\begin{itemize}
\item[-] The case $q=-1$.
\item[-]The case where $q$ is a root of unity of even order $\ell=2m$, $m \neq 1$.
\end{itemize}

To find the quantum subgroups of $SL_q(2)$, we use that $\Oc_q(SL_2(\C))$ fits in an exact sequence of Hopf algebras $\mathcal{O}(\Gamma) \hookrightarrow \mathcal{O}_q(SL_2(\C)) \twoheadrightarrow \overline{\mathcal{O}_q(SL_2(\C))}$, with $\Gamma=SL_2(\C)$ or $PSL_2(\C)$. This has been studied by many authors, and the results, that depend on the order of $q$, are shown in Theorem \ref{teoremase}.

In Sections \ref{sectionodd}, \ref{constructionodd} and \ref{determinationodd} we recall the results obtained in \cite{AG} and \cite{Mu1} for the case when the order of $q$ is odd. In Section \ref{section-1}, we compute all the quantum subgroups of $\Oc_{-1}(SL_2(\C))$, which main result is in the following theorem. \\

\noindent{\bf Theorem \ref{teo-1}.}
Let $A$ be a  Hopf algebra quotient of $\Oc_{-1}(SL_2(\C))$. Then either
    \begin{enumerate}
        \item [(i)] $A$ is isomorphic to a quotient  
        $\mathcal{O}_{-1}(SL_2(\C))/(\mathcal{J})$ where $\mathcal{J}=\ker\sigma^t$ and $\sigma: \Gamma \hookrightarrow PSL_2(\C)$ is a monomorphism of algebraic groups, or
        \item [(ii)] $A$ is isomorphic to 
        a function algebra over some dihedral group $D_{2m}$.
        \end{enumerate}

\bigskip

 Finally, in Sections \ref{sectioneven}, \ref{constructioneven} and \ref{determinationeven} we construct and determine all the quotients of $\Oc_q(SL_2(\C))$ when the order of $q$ is $\ell=2m$ with $m \neq 1$. The quotients are in bijection with even subgroup data, whose definition is as follows. \\

\noindent{\bf Definition \ref{def:subgroup-data-qeven}.}
    Let $q$ be a root of unity of even order $\ell\neq 2$. An even subgroup datum is a collection $\mathcal{D} = (I_{+}, I_{-}, N, \Gamma, \sigma, \delta)$ such that:

\begin{itemize}
    \item $I_{+}, I_{-} \in \{\emptyset, \{1\}\}$. 
    
    \item $N$ is a subgroup of $(\mathbb{Z}_{\ell})^{s}$, where $s = 1 - |I_{+} \cup I_{-}|$.
    
    \item $\Gamma$ is an algebraic group.
    
    \item $\sigma \colon \Gamma \to \bar{L}$ is an injective group morphism, where $\bar{L} \subseteq PSL_2(\C)$ is determined by the sets $I_{+}$ and $I_{-}$. 
    
    \item $\delta \colon N \to \widehat{\Gamma}$ is a group morphism.
\end{itemize} 

\bigskip

We introduce the notion of quantum subgroups isomorphism and subgroup data equivalence.
Here $q$ is a root of unity of any order.

\begin{defi} \label{isoquot}
    Let $\rho_1:\Oc_q(SL_2(\C)) \twoheadrightarrow A_1$ and $\rho_2: \Oc_q(SL_2(\C)) \twoheadrightarrow A_2$ be two quotients of $\Oc_q(SL_2(\C))$. We say that $A_1$ and $A_2$ are isomorphic as quantum subgroups if there exists $\varphi:A_1 \to A_2$ isomorphism of Hopf algebras such that $\varphi \circ \rho_1=\rho_2$.
\end{defi}

\begin{defi} \cite[ Definition 2.19]{AG}\label{isodata} Let $\mathcal{D}_1 = (I_{+}^1, I_{-}^1, N_1, \Gamma_1, \sigma_1, \delta_1)$ and $\mathcal{D}_2 = (I_{+}^2, I_{-}^2, N_2, \Gamma_2, \sigma_2, \delta_2)$ be two subgroup data. We say that $\mathcal{D}_1$ and $\mathcal{D}_2$ are equivalent, and denote $\mathcal{D}_1 \cong \mathcal{D}_2$, if 
\begin{itemize}
\item $I_{+}^1 = I_{+}^2$ and $I_{-}^1=I_{-}^2$,
\item there exists an isomorphism of algebraic groups $\sigma: \Gamma_2 \to \Gamma_1$ such that $\sigma_1  \sigma = \sigma_2$,
\item $N_1=N_2$ and $\delta_2=\sigma^t  \delta_1$.
\end{itemize}
\end{defi}

The main result of the section is the following. \\

\noindent{\bf Theorem \ref{teorema}.}
Let $q$ be a root of unity of even order $\ell =2m , m \neq 1$. There exists a bijection between
    \begin{enumerate}
        \item[(i)]  Hopf algebra quotients $\rho:\Oc_q(SL_2(\C)) \to A$ up to isomorphism.
        \item[(ii)] Even subgroup datum up to equivalence.
    \end{enumerate}

\bigskip

We would like to thank J. Bichon, W. Ferrer, C. Negron and S. Lentner for their valuable comments and contributions.

\section{Preliminaries}

In this section we give some definitions and basic results that we will be used later on. We work over the field of complex numbers $\C$.
 
\subsection{Hopf algebras} 
 For the theory of 
 Hopf algebras we refer to \cite{Mo,Ra}.
 The coproduct is denoted by 
 $ \Delta $, the counit by $\varepsilon$ and the antipode by
 $S$. As usual, we use Sweedler's 
 notation for the coproduct, i.e.  
 $\Delta(x) = x_{(1)} \otimes x_{(2)}$.
 We write  $ G(H)$ for the set of group-like elements and
 $H^+ := \ker(\varepsilon)$ for the augmentation ideal.  

\smallbreak
Following \cite{Andruskiewitsch}, a sequence of Hopf algebras maps 
\[
  B  \stackrel{\iota}{\hookrightarrow} A \stackrel{\pi}{\twoheadrightarrow} H   
\]
is called  {\it exact\/}  if  $ \iota $  is injective,
$ \pi $  is surjective, $\pi\circ \iota=\varepsilon$,  
$ \ker(\pi) = AB^+ \, $  and  
$B = {}^{\co \pi} A = \{x \in A : (\pi \otimes \id) \Delta(x)= 1 \otimes x\}$. 
We say that the sequence is \textit{central}, respec. \textit{normal}, 
if the image of $B$ is central, respect. normal, in $A$.

\begin{remark} \label{prop4mu}
Let $B  \stackrel{\iota}{\hookrightarrow} A \stackrel{\pi}{\twoheadrightarrow} H $ 
 be an exact sequence of  Hopf algebras. If $H=\C$ is the trivial Hopf algebra, then $A \cong B$.
\end{remark}

In general, it is not so straightforward to determine when a sequence is exact;
in particular, to set the equality $B = {}^{\co \pi} A$ in case of infinite-dimensional
Hopf algebras.
In this direction, the following result is quite helpful.

 \begin{prop} \label{ff} 
 Let $B \hookrightarrow A$ be an injective morphism of Hopf algebras.
 \begin{enumerate} 
\item[$(i)$] \cite[(3.12)]{ARKHIPOV}
If $B$ is commutative, then $A$ is faithfully flat as a $B$-module.
\item[$(ii)$] \cite[Proposition 3.4.3]{Mo} Assume $A$ is 
faithfully flat as a $B$-module, $AB^+=B^+ A$ and set $H=A/AB^+$ with $\pi: A \twoheadrightarrow A/AB^+$. 
Then
$B=A^{\co \pi}=\ {}^{\co \pi}A$ and $B$ is a normal Hopf subalgebra of $A$.
\end{enumerate}
\end{prop}

%
%
%






%
We now present some results on exact sequences of Hopf algebras 
from \cite{Mu1} and \cite{AG} that will be useful later. These are stated for central
extensions, but actually they hold for normal extensions. The proof follows mutatis mutandis and is left for the reader.


\begin{prop} \cite[(3.4)]{Mu1} \label{prop1mu}
    Given the normal exact sequence of Hopf algebras  
    $B  \stackrel{\iota}{\hookrightarrow} A \stackrel{\pi}{\twoheadrightarrow} H $, 
    $J$ a Hopf ideal of $A$  and $\mathcal{J}=B \cap J$, then the quotient Hopf algebra $A/J$ is contained in an exact sequence of Hopf algebras

\[
B/\mathcal{J}  \hookrightarrow A/J \twoheadrightarrow H/\pi(J) 
\]
    
\end{prop}

\begin{prop} \cite[(2.10)]{AG} \label{prop2mu}
    Let $A$ and $C$ be Hopf algebras, let $B$ be a normal Hopf subalgebra of $A$ such that $A$ is left or right faithfully flat over $B$ 
    and let $r:B \rightarrow C$ be a surjective Hopf algebra map. Then $H=A/AB^{+}$ is a Hopf algebra and $A$ fits into the exact sequence 
    $  B  \stackrel{\iota}{\hookrightarrow} A \stackrel{\pi}{\twoheadrightarrow} H$. 
    If we set $\mathcal{J}=\ker r \subseteq B$, then $(\mathcal{J})=A\mathcal{J}$ is a Hopf ideal of $A$ and 
    $A/(\mathcal{J})$ is the pushout given by the following diagram.

    $$\xymatrix {
    B \ar@{^{(}->}[r]^{\iota} \ar@{->>}[d]^{r} 
    & A \ar@{->>}[d]^{\rho} \\
    C \ar@{^{(}->}[r]^{\bar{\iota}}  & A/(\mathcal{J}) 
    }$$

    Moreover, $C$ can be identified with a normal 
    Hopf subalgebra of $A/(\mathcal{J})$ and $A/(\mathcal{J})$ fits into the exact sequence

    $$  C  \stackrel{\bar{\iota}}{\hookrightarrow} A/(\mathcal{J}) \stackrel{\bar{\pi}}{\twoheadrightarrow} H $$

\end{prop}



    

\subsection{$SL_{2}(\C)$, $\mathfrak{sl}_{2}(\C)$,
$U(\mathfrak{sl}_{2})$ and $\Oc(SL_2(\C))$}
Throughout the paper we will work with 
the groups $SL_{2}(\C)$, $PSL_{2}(\C)$, and the 
Lie algebra $\mathfrak{sl}_{2}(\C)$. 
We recall shortly the definition
to set notation. 

The group $SL_2(\C)$ is the algebraic group 
defined as the subgroup of 
$GL_2(\C)$ of matrices of determinant $1$. Its 
Lie algebra $\mathfrak{sl}_{2}(\C)$ consists of 
$2\times 2$ matrices with entries in $\C$ and zero trace with 
bracket given by the commutator \([X,Y]:=XY-YX\). 
Equivalently, it can be presented as the $\C$-vector space
spanned by the elements $h,e,f$ satisfying the bracket relations:
\begin{equation*}
    [h,e]=2e,\quad [e,f]=h,\quad [h,f]=-2f.
\end{equation*}
The isomorphism between the two presentations is given by the 
standard representation under the identification
    \[e=\left(\begin{array}{cc}
    0 & 1 \\
    0 & 0
    \end{array}\right),\qquad 
    h=\left(\begin{array}{cc}
    1 & 0 \\
    0 & -1
    \end{array}\right),\quad 
    f=\left(\begin{array}{cc}
    0 & 0 \\
    1 & 0
    \end{array}\right).\]
The universal enveloping algebra $U(\mathfrak{sl}_{2}(\C))$ 
is the unital associative $\C$-algebra generated by $e,h,f$ modulo the relations:
$$
        he - eh = 2e, \ \ \
        hf - fh = -2f, \ \ \
        ef - fe = h.
$$
%

The coordinate algebra $\Oc(SL_2(\C))$ is defined as the commutative algebra generated 
by the elements $X_{11}, X_{12}, X_{21}, X_{22}$ satisfying the relation $X_{11}X_{22}-X_{12}X_{21}=1$. That is,

\[
\Oc(SL_2(\C)) = \C  [X_{11}, X_{12}, X_{21}, X_{22} \ |\ X_{11}X_{22}-X_{12}X_{21}=1 ]. 
\] 
\subsection{On $PSL_2(\C)$ and polyhedral groups}

Recall that $PSL_2(\C)$ is the algebraic group  
defined as the quotient of $SL_2(\C)$ by the finite subgroup 
$\{\pm I\}$. Its Lie algebra is also $\mathfrak{sl}_{2}(\C)$.
Also, as an algebraic group it is 
isomorphic to $SO_3(\C)=\{A \in SL_3(\C) | \,A^{t}A=Id\}$.

The following lemma gives 
the coordinate algebra of $PSL_2(\C)$.
Although it is a well-know result, see \cite[Exercise (I.7.M)]{Ken}, we include it 
here for completeness.

\begin{lemma}
The coordinate algebra of $PSL_2(\C)$ is the subalgebra of $\Oc(SL_2(\C))$ generated by the elements
$X_{ij}X_{kl}\,$ for all $ i,j,k,l \in \{1,2\}$.
\end{lemma}

\begin{proof} The exact sequence of algebraic groups
$$
\xymatrix{
\{\pm I\} 
\ar@{^{(}->}[r]^{\iota}& SL_2(\C)\ar@{->>}[r]^{\pi}&
PSL_2(\C)} 
$$
induces the exact sequence of Hopf algebras
$$
\xymatrix{
\mathcal{O}(PSL_2(\C)) \ar@{^{(}->}[r]^{\pi^*}&
\mathcal{O}(SL_2(\C)) \ar@{->>}[r]^{\iota^*}&
\mathcal{O}(\{\pm Id \}).
} 
$$
Then \begin{small}$$\mathcal{O}(PSL_2(\C)) = \mathcal{O}(SL_2(\C)/\{\pm I\})
=\mathcal{O}(SL_2(\C)) ^{\co \iota^*} 	
= \{ \alpha \in \mathcal{O}(SL_2(\C))\, |\, (\iota^{*} \otimes id) 
\Delta(\alpha)= 1 \otimes \alpha \}.$$\end{small}
In particular, for $a \in \{\pm I \}, b \in SL_2(\C)$ and $\alpha \in \Oc(SL_2(\C))$, we have 
\[
(\iota^{*} \otimes id) \Delta(\alpha)(a \otimes b)= \Delta (\alpha) (\iota(a) \otimes b)=\alpha (\iota(a)b)
\]
and $(1 \otimes \alpha)(a \otimes b)=1(a)\alpha(b)=\alpha(b)$. 

So, $\mathcal{O}(PSL_2(\C))= \{\alpha \in \mathcal{O}(SL_2(\C)) \ | \  \alpha(\iota(a)b)=\alpha(b),\ \ \forall a \in \{\pm Id \}, b \in SL_2(\C)  \}.$ 
 Take $a=-Id$ and consider
$\alpha = \displaystyle\sum_{l} \alpha_l \ X_{i_1j_1} X_{i_2j_2} \ldots X_{i_kj_k} \in \mathcal{O}(SL_2(\C))$. 
Then, for $b \in SL_2(\C)$ we have that 
$\alpha(b) = \displaystyle\sum_{l} \alpha_l \ b_{i_1j_1} b_{i_2j_2} \ldots b_{i_kj_k}$ and 
\[
\alpha(-b) = \displaystyle\sum_{l} \alpha_l \ (-b_{i_1j_1}) (-b_{i_2j_2}) \ldots (-b_{i_kj_k}) = \displaystyle\sum_{l} \alpha_l \ (-1)^k \ b_{i_1j_1} b_{i_2j_2} \ldots b_{i_kj_k}
\]
As $a=-Id$, then $\alpha(b)=\alpha(-b)$ implies $(-1)^k = 1$ so $k$ is even. \\

In conclusion, $\mathcal{O}(PSL_2(\C)) = \C[X_{ij}X_{kl} \ |\ i,j,k,l \in \{1,2\}]$ as we wanted to prove. 

\end{proof}

In Sections $\ref{sec2.4}, \ref{sec2.5}$ and $\ref{sec2.6}$ we will define the quantum groups 
$\Uc_q(\mathfrak{sl}_n)$, $\uc_q(\mathfrak{sl}_n)$ and  $\Oc_q(SL_n(\C))$ for $q$ indeterminate, that in Section $3$ will be replaced by roots of unity. 
We will follow the definitions from \cite{Lu2}, \cite{Takeuchi} and \cite{Mu1}.

\subsection{Quantum enveloping algebras  $\Uc_q(\mathfrak{sl}_n)$, for $q$ indeterminate} \label{sec2.4}

Let $K=\mathbb{Q}(q)$ with $q$ an indeterminate. Let $U_q(\mathfrak{sl}_n)_K$ be the $K$-algebra defined by 
generators $k_i, k_i^{-1}$ $(1 \leq i \leq n)$, $e_j, f_j$ $(1 \leq j < n)$ and relations

\begin{align*}
    k_i k_i^{-1} &= k_i^{-1} k_i = 1, \quad k_i k_j = k_j k_i,  \\
    k_i e_j k_i^{-1} &= q^{2\delta_{ij}-\delta_{i+1,j}-\delta_{i,j+1}} e_j, \quad k_i f_j k_i^{-1} = q^{-2\delta_{ij}+\delta_{i+1,j}+\delta_{i,j+1}} f_j,  \\
    e_j f_l - f_l e_j &= \delta_{j,l} \frac{k_j - k_j^{-1} }{q^{-1} - q},  \\
    e_j e_l &= e_l e_j, \quad f_j f_l = f_l f_j \quad \text{if } |j - l| > 1,  \\
    e_j^2 e_{l} - &(q + q^{-1}) e_j e_{l} e_j + e_{l} e_j^2 = 0, \\
    f_j^2 f_{l} - &(q + q^{-1}) f_j f_{l} f_j + f_{l} f_j^2 = 0 \quad \text{if } |j - l| = 1. 
\end{align*}

with $1 \leq i, k \leq n, \ 1 \leq j, l < n$. There is a Hopf algebra structure on $U_q(\mathfrak{sl}_n)_K$ where

\begin{align*}
    \Delta(k_i) = k_i \otimes k_i, \quad \quad  
    \Delta(e_j) = 1 \otimes e_j &+ e_j \otimes k_j , \quad \quad  
    \Delta(f_j) = k_j^{-1} \otimes f_j + f_j \otimes 1,\\
    \varepsilon(k_i) = 1, \quad &\varepsilon(e_j) = \varepsilon(f_j) = 0, \\
    S(k_i) = k_i^{-1}, \quad S(e_j) &= -e_j k_j^{-1} , \quad S(f_j) = -k_j f_j.
\end{align*} 

\

The Hopf algebra $U_q(\mathfrak{sl}_n)_K$ has the following Lusztig form $\Uc_q(\mathfrak{sl}_n)$, 
see \cite[(3.2)]{Takeuchi}.

\begin{defi} \label{defUq1}
    Let $R=\Z[q, q^{-1}]$. We define $\mathcal{U}_q(\mathfrak{sl}_n)_R$ as the $R$-subalgebra of $U_q(\mathfrak{sl}_n)_K$ generated by $k_i$, 
    $\left[ \begin{array}{c} k_i \ ; 0 \\ t \end{array} \right]$, 
    $e_j^{(N)}$, $f_j^{(N)}$ $(1 \leq i \leq n, \ 1 \leq j < n, \ 0 \leq N,t)$, where

\begin{center}
\[
\left[ \begin{array}{c} k_i \ ; 0 \\ t \end{array} \right] =
\prod_{s=1}^{t} \frac{q^{-s+1}k_i - q^{s-1}k_i^{-1}}{q^{s} - q^{-s}}
\]
\end{center}

\begin{align*}
    e_j^{(N)} &= \frac{e_j^N}{[N]!}, \quad f_j^{(N)} = \frac{f_j^N}{[N]!}, \quad \text{with} \quad [N]! = \prod_{s=1}^{N} \frac{q^{s} - q^{-s}}{q - q^{-1}}. 
    \end{align*}
\end{defi}

$\mathcal{U}_q(\mathfrak{sl}_n)_R$  is a Hopf algebra integral form of  $U_q(\mathfrak{sl}_n)_K$ with

\[ \Delta \left[ \begin{array}{c} k_i \ ; 0\\ t \end{array} \right] = \sum_{\nu=0}^{t} k_i^{\nu-t} 
\begin{bmatrix} k_i \ ; 0 \\ \nu \end{bmatrix} \otimes k_i^{\nu} \begin{bmatrix} k_i \ ; 0 \\ t - \nu \end{bmatrix}, \] 
\[
\Delta(e_j^{(N)}) = \sum_{\nu=0}^{N} q^{\nu(N-\nu)} e_j^{(N-\nu)}k_j ^{\nu} \otimes  e_j^{(\nu)},
\]

\[
\Delta(f_j^{(N)}) = \sum_{\nu=0}^{N} q^{-\nu(N-\nu)} f_j^{(\nu)} \otimes k_i^{-\nu} f_j^{(N-\nu)}.
\]

\subsection{Quantum groups  $\Oc_q(SL_n(\C))$, for $q$ indeterminate} \label{sec2.5}

\begin{defi}
    Let $K=\Q(q)$. The standard quantization of the algebra of functions on $n \times n$ matrices $\Oc_q(M_n)_K$ is the $K$-algebra generated by the elements $x_{ij}$ for $1\leq i,j\leq n$, subject to the $q$-matrix relations
    \begin{align*}
    x_{is}x_{js} &= q\,  x_{js}x_{is},  & x_{si}x_{sj}& = q\, x_{sj}x_{si} & \text{if } i<j, \\
    x_{it}x_{js}& =x_{js}x_{it}, & x_{is}x_{jt}-x_{jt}x_{is} & =(q-q^{-1})x_{it}x_{js}, & \text{if } i<j, s<t. 
    \end{align*}

\end{defi}

\begin{defi} \cite{Ken}
    Let $K=\Q(q)$. We define $\Oc_q(SL_n)_K$ as the $K$-algebra given by the quotient $\Oc_q(SL_n)_K:=\Oc_q(M_n)_K/(det_q-1)$, where $(det_q-1)$ is the two sided ideal of $\Oc_q(M_n)_K$ generated by the element $det_q-1 = \displaystyle\sum_{\pi \in S_n} (-q)^{l(\pi)}x_{1\pi(1)}x_{2\pi(2)} \ldots x_{n\pi(n)} -1$.

A Hopf algebra structure is given by
\[ \Delta (x_{ij})  = \sum_{s=1}^{n} x_{is} \otimes x_{sj}, \] 
\[  \S (x_{ij}) = (-q)^{(i-j)} A_{ji}  \]
\[ \varepsilon (x_{ij}) = \delta_{ij} \]

where $A_{ji}$ denotes the $j,i$-th quantum minor of $\Oc_q(M_n)_K$. For more details see \cite{Parshall}.
\end{defi}

\begin{example} \label{Oq(sl2)}
The quantum special linear group  
$\Oc_q(SL_2)_K$ is the unitary associative $K$-algebra 
generated by the elements $a,b,c,d$
satisfying the following relations 
\begin{align*}
ab &= q\,  ba,  & bc& = cb, &  ac &= q  ca,  & ad-da = (q-q^{-1})  bc, \\
bd& =q\, db, & cd &=q dc, &  ad-q bc& =1. & 
\end{align*}

The coalgebra structure and antipode are given by 
\begin{align*}
\triangle(a) &= a \otimes a + b \otimes c,  & \varepsilon(a)& = 1, &  \S(a) &= d,  \\
\triangle(b)& =a \otimes b + b \otimes d, & \varepsilon(b)& = 0, &  \S(b) &= -qb \\
\triangle(c) &= c \otimes a + d \otimes c,  & \varepsilon(c)& = 0, &  \S(c) &= -q^{-1}c, \\
\triangle(d) &= c \otimes b + d \otimes d,  & \varepsilon(d)& = 1, &  \S(d) &= a &
\end{align*}
\end{example}

\begin{remark} 
    \textbf{(Specialization at roots of unity).} 
    In \cite[(III.7.1.)]{Ken}, Brown and Goodearl studied the specialization at roots of unity of the algebras $\Uc_q(\mathfrak{g})$ and $\Oc_q(G)$ for $G$ a connected, simply connected complex semisimple group with Lie algebra $\mathfrak{g}$. It goes as follows:
    
    Let $R=\Z[q,q^{-1}]$ and $\epsilon \neq \pm1$ be a complex primitive $\ell^{th}$ root of unity. Let $\chi_{\ell}(q) \in R$ be the $\ell^{th}$ cyclotomic polynomial, so that $R/\chi_{\ell}(q)R\cong \Q(\epsilon)$ and $\chi_{\ell}(q)\Oc_q(G)$ is a Hopf ideal of $\Oc_q(G)$. The algebra $\Oc_q(G)/\chi_{\ell}(q)\Oc_q(G)$ is denoted $\Oc_{\epsilon}(G)_{\Q(\epsilon)}$ and is called the quantized coordinate algebra of $G$ over $\Q(\epsilon)$ at the root of unity $\epsilon$. In the same way, we can form the $\Q(\epsilon)$- Hopf algebra $\Uc_{\epsilon}(\mathfrak{g}):=\Uc_q(\mathfrak{g})/\chi_{\ell}(q)\Uc_q(\mathfrak{g})$ and the Hopf pairing of $(2.6.2)$ induces a Hopf pairing between $\Oc_{\epsilon}(G)_{\Q(\epsilon)}$ and $\Uc_{\epsilon}(\mathfrak{g})$. If $\Bbbk$ is any field containing $\Q(\varepsilon)$ we can obtain a $\Bbbk$-form of $\Oc_{\epsilon}(G)$, namely $\Oc_{\epsilon}(G)_{\Bbbk}:= \Oc_{\epsilon}(G) \otimes_{\Q(\epsilon)}\Bbbk$. When $\Bbbk=\C$ we simply write $\Oc_{\epsilon}(G)$. 
\end{remark}

\begin{nota}
{\it    From now on we set $q$ to be an $\ell-th$ root of 
unity with $q^2\neq 1$, and $\Bbbk=\C$. Let $m=\ord(q^2)$, so if $\ell$ is odd $\ell=m$ and if $\ell$ is even $\ell=2m$}. 
\end{nota}

\subsection{Small quantum groups} \, \label{sec2.6}

 In this section we  introduce the algebras $\uc_{2,q}(\mathfrak{sl}_n)$ and $\uc_q(\mathfrak{sl}_n)$ 
 defined by Takeuchi in \cite{Takeuchi} who denoted them as  
 $U_q(\mathfrak{sl}_n)^{[\ell]}$ and $U_q(\mathfrak{sl}_n)^{[[\ell]]}$, respectively. 
 We will follow the notation given in \cite{Mu1} and  \cite{Chelsea}. 
 Then, we characterize the duals of the Frobenius–Lusztig kernels 
 $u_{2,q}(\mathfrak{sl_n})^{\circ}$ and $u_q(\mathfrak{sl}_n)^{\circ}$. 

 First, we begin by defining the algebras $u_{2,q}(\mathfrak{sl}_n)$ for arbitrary $\ell$.

\begin{defi} \label{defu2q}
           We define $u_{2,q}(\mathfrak{sl}_n)$ as the Hopf subalgebra of $\Uc_q(\mathfrak{sl}_n)$ 
        generated by $k_i$, $k_i^{-1}$, $e_i$ and $f_i$ $(1 \leq i < n)$. 
        Here, $k_i$ is an element of order $2\ell$ and $e_j,f_j$ are elements of order $\ell$.
        The dimension of this Hopf algebra is $2^{n-1}\ell^{n^2-1}$. 
\end{defi}

We will use the notation $u_q(\mathfrak{sl}_n)$ for the case when the order of $q$ is odd and the definition is as follows.

\begin{defi} \label{defuq}
 We define $$u_q(\mathfrak{sl}_n):=u_{2,q}(\mathfrak{sl}_n)/(k_i^{\ell}-1),$$ which is a  Hopf algebra of dimension $\ell^{n^2-1}$.
 \end{defi}

Now we define two quotients of  $\Oc_q(GL_n)$, one for the case when the order of $q$ is arbitrary and the other for the case when ord($q$) is odd, and it turns out that both quotients are isomorphic to the duals of the algebras defined in Definitions \ref{defu2q} and \ref{defuq}.

\begin{defi} \label{overline} Let $q$ be a root of unity with arbitrary order $\ell$ and $m=\ord(q^2)$.
We define $\overline{\Oc}_{q}(SL_{n}(\C))$ as the quotient of $\Oc_{q}(M_{n}(\C))$ by the ideal $(det_q-1, x_{ii}^{2m}-1, x_{ij}^{m} (i \neq j))$. The dimension is $2^{n-1}m^{n^2-1}$.

\end{defi}

\begin{defi} \label{widehat} Let $q$ be a root of unity with odd order $\ell>2$.
We define $\widehat{\Oc}_{q}(SL_{n}(\C))$ as the quotient of $\Oc_{q}(M_{n}(\C))$ by the ideal $(det_q-1, x_{ii}^{\ell}-1, x_{ij}^{\ell} (i \neq j))$. The dimension is $\ell^{n^2-1}$.
\end{defi}

From now on we will work in the case $n=2$. We start giving the definition of $\Oc_{-1}(SL_2(\C))$.

\begin{defi} \cite[(5.2)]{Bichon}
    The algebra $\Oc_{-1}(SL_2(\C))$ is generated by $a,b,c,d$ with relations 
    \begin{align*}
    ab &= -  ba,  & bc& = cb, &  ac &= - ca,  & ad = da, \\
    bd& =- db, & cd &=- dc, &  det _{-1} &= ad+ bc =1. 
    \end{align*}

$\Oc_{-1}(SL_2(\C))$ is a Hopf algebra with the coproduct, counit and antipode defined in Example \ref{Oq(sl2)}. 
\end{defi}

The Hopf algebra $\Oc_{q}(SL_{2}(\C))$ has a central or normal commutative 
Hopf subalgebra depending on the order of $q$, as we can see in the next lemma.

\begin{lemma}
\begin{itemize}
    \item[(i)] Assume that $q$ is a root of unity of odd order $\ell>2$. Consider the subalgebra $L$ of $\Oc_q(SL_2(\C))$ generated by the elements $a^{\ell}, b^{\ell}, c^{\ell}$ and  $d^{\ell}$. Then, $L$ is a  central Hopf subalgebra of $\Oc_q(SL_2(\C))$ isomorphic to $\Oc(SL_2(\C))$ as Hopf algebras.
    \item[(ii)] For $q=-1$, $B= \C\langle a^2, b^2, c^2, d^2, ab, ac, bc, bd, cd \rangle$ is a commutative normal Hopf subalgebra of $\Oc_{-1}(SL_2(\C))$ isomorphic to $\mathcal{O}(PSL_2)$ as Hopf algebras.
    \item[(iii)] Assume that $q$ is a root of unity with $q^2\neq 1$ and $2m=\ell=$ ord$(q)$ even. Consider the subalgebra $N$ of $\Oc_q(SL_2(\C))$ generated by the monomials $x^my^m$ where $x,y \in \{a,b,c,d\}$. Then, $N$ is a commutative normal Hopf subalgebra of $\Oc_q(SL_2(\C))$ isomorphic to $\Oc(PSL_2(\C))$ as Hopf algebras.
\end{itemize}
\end{lemma}

\begin{proof}
    \begin{itemize}
    \item[(i)] The case when the order of $q$ is odd has been studied in \cite[Proposition III.3.1]{Ken}. The isomorphism is given by $X_{11} \to a^{\ell}, X_{12} \to b^{\ell}, X_{21} \to c^{\ell}$ and $X_{22} \to d^{\ell}$.
    \item[(ii)] The isomorphism is given by
        \begin{center}
        $\mathcal{O}(PSL_2(\C)) \stackrel{\varphi}{\longrightarrow} \B  $\\
        $\begin{pmatrix}
        X_{11}^2 & X_{11}X_{12} & X_{12}^2\\
        X_{11}X_{21} & X_{12}X_{21} & X_{12}X_{22}\\
        X_{21}^2 & X_{21}X_{22} & X_{22}^2
        \end{pmatrix} \mapsto \begin{pmatrix}
        a^2 & ab & -b^2\\
        ac & -bc & -bd\\
        -c^2 & -cd & d^2
        \end{pmatrix}.$
        \end{center}
where each entry of the left hand matrix maps to the same entry in the right hand matrix.
    \item[(iii)] The case when the order of $q$ is even but $q^2\neq 1$ follows from \cite[Lemma 6.15]{Chelsea}.
    \end{itemize}
\end{proof}

\begin{remark}
    Each subalgebra $L, B, N$ defines a Hopf ideal $\Oc_q(SL_2(\C))L^{+}, \Oc_q(SL_2(\C))B^{+}$ and $\Oc_q(SL_2(\C))N^{+}$ that coincides with those given in Definitions 2.6.3 and 2.6.4, respectively.
\end{remark}

It is interesting to think about whether $\Oc_q(SL_2(\C))$ fits in an exact sequence of Hopf algebras
\[
\mathcal{O}(\Gamma) \hookrightarrow \mathcal{O}_q(SL_2(\C)) \twoheadrightarrow \overline{\mathcal{O}_q(SL_2(\C))},
\]

with $\Gamma=SL_2(\C)$ or $PSL_2(\C)$. This has been studied by many authors, and the results, which depend on the order of $q$, are shown in the following theorem.

\begin{thm} \label{teoremase}
    \begin{itemize}
        \item[(i)] \cite{Mu1} If the order of $q$ is odd, $\overline{\mathcal{O}_q(SL_2(\C))} \cong \mathfrak{u}_q(\mathfrak{sl}_2)^{\circ}$ and $\Oc_q(SL_2(\C))$ fits in the exact sequence

        \bigskip
    
        \begin{center}
        $\mathcal{O}(SL_2(\C)) \hookrightarrow \mathcal{O}_q(SL_2(\C)) \twoheadrightarrow \mathfrak{u}_q(\mathfrak{sl}_2)^{\circ}. $
        \end{center} 
 Here 
        \begin{align*}
            \overline{\Oc_{q}(SL_{2}(\C))} & =\Oc_q(SL_2(\C))/(\mathcal{O}_q(SL_2(\C)) L^{+}) \\
	& = \Oc_q(SL_2(\C))/(a^{\ell}-1,b^{\ell},c^{\ell},d^{\ell}-1) \\
	& =\widehat{\Oc}_{q}(SL_2(\C)).
        \end{align*}
  (See Definition \ref{widehat})
        
        \bigskip
        
        \item[(ii)]\cite{Bichon} When $q=-1$, $\overline{\mathcal{O}_{-1}(SL_2(\C))} \cong \C \Z_2$ and $\Oc_{-1}(SL_2(\C))$ fits in the exact sequence

        \bigskip
        
        \begin{center}
        $\mathcal{O}(PSL_2(\C)) \hookrightarrow \mathcal{O}_{-1}(SL_2(\C)) \twoheadrightarrow \C \Z_2 $.
        \end{center} 
        \item[(iii)] \cite{Chelsea} If ord$(q)=2m > 2 $ is even,  $\overline{\mathcal{O}_q(SL_2(\C))}
        \cong \mathfrak{u}_{2,q}(\mathfrak{sl}_2)^{\circ}$ and the exact sequence is

        \bigskip
        
        \begin{center}
        $\mathcal{O}(PSL_2(\C)) \hookrightarrow \mathcal{O}_q(SL_2(\C)) \twoheadrightarrow \mathfrak{u}_{2,q}(\mathfrak{sl}_2)^{\circ}. $
        \end{center}

        \bigskip

Here
        
        \begin{align*}
           \overline{\Oc_{q}(SL_{2}(\C))} & = \mathcal{O}_q(SL_2(\C))/(\mathcal{O}_q(SL_2(\C)) N^{+}) \\
	& = \mathcal{O}_q(SL_2(\C))/(a^{2m}-1,b^m,c^m,d^{2m}-1) \\
	& = \bar{\Oc}_{q}(SL_{2}(\C)).
        \end{align*}
        (See Definition \ref{overline})
    \end{itemize}
\end{thm}


\section{Quantum subgroups of $\mathcal{O}_q(SL_2(\C))$} \label{section3}

We want to compute the quantum subgroups of $\mathcal{O}_q(SL_2(\C))$ for different values of the parameter $q$. That is, find the quotients of Hopf Algebras for $\mathcal{O}_q(SL_2(\C))$. If $\rho: \mathcal{O}_q(SL_2(\C)) \twoheadrightarrow A$ is an epimorphism of Hopf algebras, we first construct the following commutative diagram

$$\xymatrix {
    B \ar@{^{(}->}[r]^(.3){\iota} \ar@{->>}[d]^{p} 
    & \mathcal{O}_q(SL_2(\C)) \ar@{->>}[r]^(.6){\pi} \ar@{->>}[d]^{\rho} & H \ar@{->>}[d]^{r} \\
    \mathcal{O}(\Gamma) \ar@{^{(}->}[r]^(.5){\bar{\iota}}  & A \ar@{->>}[r]^(.5){\bar{\pi}} & \overline{H}
    }$$

in which the rows are exact sequences. We will apply similar strategies as Müller in \cite{Mu1} to find the quotients. For a more general approach, see \cite{AG}. \\

\begin{prop} \label{propiso}
    Consider the commutative diagram of Hopf algebras whose rows are exact sequences

$$\xymatrix {
    \Oc(\Gamma) \ar@{^{(}->}[r]^(.4){\iota} \ar@{=}[d]  & \Oc_q(SL_2(\C)) \ar@{->>}[r]^(.6){\pi} \ar@{->>}[d]^{\rho} & H \ar@{=}[d] \\
    \Oc(\Gamma) \ar@{^{(}->}[r]^{\overline{\iota}}  & A' \ar@{->>}[r]^{\overline{\pi}} & H
    }$$

Assume that $\dim H$ is finite. Then $\rho$ is an isomorphism.
\end{prop}

\begin{proof}
    As $\Oc(\Gamma)$ is commutative,  then by Proposition \ref{ff} $A'$ is faithfully flat over $\Oc(\Gamma)$. Since $\Oc_{q}(SL_2(\C))$ is noetherian (\cite[(I.1.18.)]{Ken}), then $A'$ is also noetherian and by the proof of \cite[(1.15)]{AG}, $\rho$ is an isomorphism.
\end{proof}

\subsection{Quantum subgroups of $\Oc_{q}(SL_{2}(\C))$, $\ord(q)=\ell$ odd}  \label{sectionodd}

Recall that if ord($q)=\ell>2$ is odd, then $\Oc_q(SL_2(\C))$ fits in the exact sequence

    \bigskip
    
    \begin{center}
    $\mathcal{O}(SL_2(\C)) \hookrightarrow \mathcal{O}_q(SL_2(\C)) \twoheadrightarrow \mathfrak{u}_q(\mathfrak{sl}_2)^{\circ} $.
    \end{center} 

M\"uller in \cite{Mu1} found all the quotients of $\Oc_q(SL_n(\C))$ with these conditions for $q$. We start this section  presenting the main result for the case $n=2$. 

\begin{defi}\label{def:subgroup-data-qodd}
    Let $q$ be a root of unity of odd order $\ell$. An odd subgroup data is a collection $\mathcal{D} = (I_{+}, I_{-}, N, \Gamma, \sigma, \delta)$ such that:

\begin{itemize}
    \item $I_{+}, I_{-} \in \{\emptyset, \{1\}\}$. 
    
    \item $N$ is a subgroup of $(\mathbb{Z}_{\ell})^{s}$, where $s = 1 - |I_{+} \cup I_{-}|$.
    
    \item $\Gamma$ is a  group.
    
    \item $\sigma \colon \Gamma \to L$ is an injective group morphism, where $L$ is a subgroup of $SL_2(\C)$ determined by the sets $I_{+}$ and $I_{-}$. 
    
    \item $\delta \colon N \to \widehat{\Gamma}$ is a group morphism.
\end{itemize}
\end{defi}

\begin{remark} \label{defl}
    In the construction of quantum subgroups, we will take quotients of $\Oc_q(SL_2(\C))$ using the following notation: $I_{+}, I_{-} \in \{\emptyset, \{1\}\}$ where the choice $I_{+}=\{\emptyset\}$ corresponds to taking the quotient by the ideal generated by $c$, $I_{-}=\{\emptyset\}$ corresponds to taking the ideal generated by $b$ and if $I_{\pm}=\{1\}$ we do not take quotients. These sets define an algebraic Lie subalgebra $\mathfrak{l}$ of $\mathfrak{sl}_{2}(\C)$ with connected Lie subgroup $L$ of $SL_2(\C)$, such that $\mathfrak{l} = \mathfrak{l}_{+} \oplus \C h \oplus \mathfrak{l}_{-}$ and $\mathfrak{l}_{+} = \C e$ or $\mathfrak{l}_{+} = 0$ and $\mathfrak{l}_{-} = \C f$ or $\mathfrak{l}_{-} = 0$. 
\end{remark}

We have the following result that characterizes all the subgroups of $\Oc_q(SL_2(\C))$ for $q$ a root of unity of odd order. We will follow the notation given in \cite{AG}, and definitions of quantum subgroup isomorphism and subgroup data equivalence given in Definition \ref{isoquot} and Definition \ref{isodata}, respectively.

\begin{thm} \cite{Mu1}
   Let $q$ be a root of unity of odd order. There exists a bijection between
    \begin{enumerate}
        \item[(i)]  Hopf algebra quotients $\rho:\Oc_q(SL_2(\C)) \twoheadrightarrow A$ up to isomorphism.
        \item[(ii)] Odd subgroup data up to equivalence.
    \end{enumerate}
\end{thm}

The procedure to prove this characterization is as follows: first, construct the quotients and then see that the construction made before is comprehensive, that is, every subgroup of $\Oc_q(SL_2(\C))$ can be constructed in this way. The proof that it is a bijection follows from \cite[ Theorem 2.20]{AG}.

\subsection{Construction of quantum subgroups.} \label{constructionodd}
For the first part, we have the following theorem.


\begin{thm} \cite[Theorem 2.17]{AG}
    For every odd subgroup data $\mathcal{D} = (I_{+}, I_{-}, N, \Gamma, \sigma, \delta)$, there exists a Hopf algebra $A_{\mathcal{D}}$ that is a quotient of $\Oc_{q}(SL_2(\C))$ and fits into the exact sequence
    $$ 1 \rightarrow \Oc(\Gamma) \xrightarrow{\tilde{\iota}} A_{\mathcal{D}} \xrightarrow{\hat{\pi}} H \rightarrow 1, $$

Moreover, $A_{\mathcal{D}}$ is given by the quotient of  $\mathcal{O}_q(L)/(\mathcal{J})$ by $J_{\delta}$ where $\mathcal{J}=\ker \sigma^t$ and  $J_{\delta}$ is the two-sided ideal generated by some grouplike elements and the following diagram of exact sequences of Hopf algebras is commutative:

$$\xymatrix {
    \mathcal{O}(SL_2(\C)) \ar@{^{(}->}[r]^{\iota} \ar@{->>}[d]^{u_L} 
    & \mathcal{O}_q(SL_2(\C)) \ar@{->>}[r]^{\pi} \ar@{->>}[d]^{v_L} & \mathfrak{u}_q(\mathfrak{sl}_2)^{\circ} \ar@{->>}[d]^{w_L} \\
    \mathcal{O}(L) \ar@{^{(}->}[r]^{\iota_L} \ar@{->>}[d]^{\sigma^t}  & \mathcal{O}_q(L) \ar@{->>}[r]^{\pi_L} \ar@{->>}[d]^{s} & \mathfrak{u}_q(\mathfrak{l})^{\circ} \ar@{=}[d] \\
    \Oc(\Gamma) \ar@{^{(}->}[r]^{\widehat{\iota}} \ar@{=}[d] & \mathcal{O}_q(L)/(\mathcal{J}) \ar@{->>}[r]^{\widehat{\pi}} \ar@{->>}[d]^{t} & \mathfrak{u}_q(\mathfrak{l})^{\circ} \ar@{->>}[d]^{w} \\
    \Oc(\Gamma) \ar@{^{(}->}[r]^{\bar{\iota}} & A_{\mathcal{D}} \ar@{->>}[r]^{\bar{\pi}} & H
    }$$
    
\end{thm}

The construction is done in three steps where the first one is as follows: Take $I_{+},I_{-} \in \{ \emptyset, \{1\}\}$. Then one can construct a Hopf subalgebra $\mathfrak{u}_q(\mathfrak{l}) \subseteq \mathfrak{u}_q(\mathfrak{sl}_2)$  and an algebraic subgroup $L$ of $SL_2(\C)$ such that $Lie(L)=\mathfrak{l}$. We have an epimorphism  $w_L : \mathfrak{u}_q(\mathfrak{sl}_2)^{\circ}
\rightarrow \mathfrak{u}_q(\mathfrak{l})^{\circ}$ of Hopf algebras and
 the following commutative
diagram of exact sequences of Hopf algebras

$$\xymatrix {
    \mathcal{O}(SL_2(\C)) \ar@{^{(}->}[r]^{\iota} \ar@{->>}[d]^{u_L} 
    & \mathcal{O}_q(SL_2(\C)) \ar@{->>}[r]^{\pi} \ar@{->>}[d]^{v_L} & \mathfrak{u}_q(\mathfrak{sl}_2)^{\circ} \ar@{->>}[d]^{w_L} \\
    \mathcal{O}(L) \ar@{^{(}->}[r]^{\iota_L}  & \mathcal{O}_q(L) \ar@{->>}[r]^{\pi_L} & \mathfrak{u}_q(\mathfrak{l})^{\circ}
    }$$
 
After this construction we obtain four possible diagrams that are shown below.

\begin{enumerate}
    \item[Case I:] Let $I_{+}=\{\emptyset\}$ and $I_{-}=\{\emptyset\}$. We take the quotient by the ideals generated by $\{b\}$ and $\{c\}$ so $\mathfrak{l} =  \C h$ and the diagram is 
    $$\xymatrix {
    \mathcal{O}(SL_2(\C)) \ar@{^{(}->}[r]^{\iota} \ar@{->>}[d]^{u_L} 
    & \mathcal{O}_q(SL_2(\C)) \ar@{->>}[r]^{\pi} \ar@{->>}[d]^{v_L} & \mathfrak{u}_q(\mathfrak{sl}_2)^{\circ} \ar@{->>}[d]^{w_L} \\
    \mathcal{O}(SL_2(\C))/( X_{12}, X_{21} ) \ar@{^{(}->}[r]^(.5){\iota_L}  & \mathcal{O}_q(SL_2(\C))/(b,c) \ar@{->>}[r]^(.6){\pi_L} & \C\Z_{\ell}
    }$$
Here $\mathcal{O}(SL_2(\C))/( X_{12}, X_{21} )$ corresponds to the function algebra on the torus given by diagonal matrices.

    \item[Case II:] Let  $I_{+}=\{1\}$ and $I_{-}=\{\emptyset\}$. We take the quotient by $(c)$ so $\mathfrak{l} = \C e \oplus \C h$ and obtain
    $$\xymatrix {
    \mathcal{O}(SL_2(\C)) \ar@{^{(}->}[r]^{\iota} \ar@{->>}[d]^{u_L} 
    & \mathcal{O}_q(SL_2(\C)) \ar@{->>}[r]^{\pi} \ar@{->>}[d]^{v_L} & \mathfrak{u}_q(\mathfrak{sl}_2)^{\circ} \ar@{->>}[d]^{w_L} \\
    \mathcal{O}(SL_2(\C))/( X_{21}) \ar@{^{(}->}[r]^{\iota_L}  & \mathcal{O}_q(SL_2(\C))/(c) \ar@{->>}[r]^(.6){\pi_L} & \mathfrak{u}_q(\mathfrak{b}^{+})^{\circ}
    }$$
Here $\mathcal{O}(SL_2(\C))/( X_{21} )$ corresponds to the function algebra on the Borel subgroup of upper triangular matrices.

    \item[Case III:] Let $I_{+}=\{\emptyset\}$ and $I_{-}=\{1\}$. We take the quotient by $(b)$ so $\mathfrak{l} =  \C h \oplus \C f$ and obtain $$\xymatrix {
    \mathcal{O}(SL_2(\C)) \ar@{^{(}->}[r]^{\iota} \ar@{->>}[d]^{u_L} 
    & \mathcal{O}_q(SL_2(\C)) \ar@{->>}[r]^{\pi} \ar@{->>}[d]^{v_L} & \mathfrak{u}_q(\mathfrak{sl}_2)^{\circ} \ar@{->>}[d]^{w_L} \\
    \mathcal{O}(SL_2(\C))/(X_{12}) \ar@{^{(}->}[r]^{\iota_L}  & \mathcal{O}_q(SL_2(\C))/(b) \ar@{->>}[r]^(.6){\pi_L} & \mathfrak{u}_q(\mathfrak{b}^{-})^{\circ}
    }$$
Here $\mathcal{O}(SL_2(\C))/( X_{12} )$ corresponds to the function algebra on the Borel subgroup of lower triangular matrices.

    \item[Case IV:] Let $I_{+}=\{1\}$ and $I_{-}=\{1\}$. We do not take quotients so $\mathfrak{l} = \C e \oplus \C h \oplus \C f$ and we have  $$\xymatrix {
    \mathcal{O}(SL_2(\C)) \ar@{^{(}->}[r]^{\iota} \ar@{=}[d]
    & \mathcal{O}_q(SL_2(\C)) \ar@{->>}[r]^{\pi} \ar@{=}[d] & \mathfrak{u}_q(\mathfrak{sl}_2)^{\circ} \ar@{=}[d] \\
    \mathcal{O}(SL_2(\C)) \ar@{^{(}->}[r]^{\iota}  & \mathcal{O}_q(SL_2(\C)) \ar@{->>}[r]^{\pi} & \mathfrak{u}_q(\mathfrak{sl}_2)^{\circ}
    }$$

\end{enumerate}

For the second step, let $\sigma: \Gamma \rightarrow SL_2(\C)$  be the inclusion of the subgroup $\Gamma$ in $SL_2(\C)$ such that $\sigma(\Gamma) \subseteq L$ and $\sigma^t : \mathcal{O}(L) \rightarrow \mathcal{O}(\Gamma)$. Then, by Proposition \ref{prop2mu} the quotients of $\mathcal{O}_q(L)$ fits in an exact sequence and we have the following commutative diagram

$$\xymatrix {
    \mathcal{O}(SL_2(\C)) \ar@{^{(}->}[r]^{\iota} \ar@{->>}[d]^{u_L} 
    & \mathcal{O}_q(SL_2(\C)) \ar@{->>}[r]^{\pi} \ar@{->>}[d]^{v_L} & \mathfrak{u}_q(\mathfrak{sl}_2)^{\circ} \ar@{->>}[d]^{w_L} \\
    \mathcal{O}(L) \ar@{^{(}->}[r]^{\iota_L} \ar@{->>}[d]^{\sigma^t}  & \mathcal{O}_q(L) \ar@{->>}[r]^{\pi_L} \ar@{->>}[d]^{s} & \mathfrak{u}_q(\mathfrak{l})^{\circ} \ar@{=}[d] \\
    \Oc(\Gamma) \ar@{^{(}->}[r]^{\widehat{\iota}} & \mathcal{O}_q(L)/(\mathcal{J}) \ar@{->>}[r]^{\widehat{\pi}} & \mathfrak{u}_q(\mathfrak{l})^{\circ}
    }$$

where $\mathcal{J}=\ker \sigma^t$ and $\Oc(\Gamma) \cong \mathcal{O}(L)/\mathcal{J}$.


\begin{example}
Let $L$ be the Borel subgroup of upper triangular matrices, that appears in Case II. Let $G_a$ be the additive group on $\C$ and consider the inclusion 
    $\sigma:G_{a} \to SL_2(\C)$ defined by 
    \[
    \sigma(X)= \begin{pmatrix}
        1 & X\\
        0 & 1
        \end{pmatrix}
    \] for all $X \in G_a$. Then we have that $\mathcal{O}_q(SL_2(\C))/\langle a^{\ell}-1, d^{\ell}-1, b^{\ell}, c\rangle$ is isomorphic to the Taft algebra $T_{q^2}$ and the following diagram is commutative:

\begin{small}
    $$\xymatrix {
    \mathcal{O}(SL_2(\C)) \ar@{^{(}->}[r] \ar@{->>}[d] 
    & \mathcal{O}_q(SL_2(\C)) \ar@{->>}[r] \ar@{->>}[d] & \mathfrak{u}_q(\mathfrak{sl}_2)^{\circ} \ar@{->>}[d] \\
    \mathcal{O}(SL_2(\C))/( X_{21}) \ar@{^{(}->}[r] \ar@{->>}[d]^{\sigma^t}  & \mathcal{O}_q(SL_2(\C))/(c) \ar@{->>}[r] \ar@{->>}[d] & \mathcal{O}_q(SL_2(\C))/\langle a^{\ell}-1, d^{\ell}-1, b^{\ell}, c\rangle \ar@{=}[d] \\
    \Oc(G_{a}) \ar@{^{(}->}[r] & \mathcal{O}_q(SL_2(\C))/\langle a^{\ell}-1, d^{\ell}-1, c^{\ell}\rangle \ar@{->>}[r] & \mathcal{O}_q(SL_2(\C))/\langle a^{\ell}-1, d^{\ell}-1, b^{\ell}, c\rangle
    }$$
\end{small}
\end{example}

For the third step, we have to consider the subgroups $N\subseteq (\mathbb{Z}_{\ell})^s$. When $s=0$, that is, cases II, III and IV, the only subgroup is $N=\{0\}$. 

In case I, $s=1$ and  the subgroups of $\Z_{\ell}$ are $N=\Z_t$ for some $t$ divisor of ${\ell}$. As we take the quotients by $(b,c)$, the element $\bar{a} \in \mathcal{O}_q(SL_2(\C))/(b,c)$ is grouplike. So, $N = \Z_t = (p)$ where $\ell = p t$. Now $\delta:\Z_t \to \widehat{\Gamma}  $ sends $p$ to some character of $\Gamma$. One may take the quotient of $\Oc_q(SL_2(\C))/(b,c)$ by $J_{\delta}=(\bar{a}^p - \delta(p))$ and get another quantum subgroup. To understand it better, let us depict this with an example.

\begin{example}
Let $G_m$ be the multiplicative group on $\C^{\times}$ and $G_k=\{g \in \C \  | \ g^k=1\}=\langle \omega\rangle \subset G_m$. Consider the monomorphisms  $\sigma_{1} :G_m \hookrightarrow SL_2(\C)$ given by $y \mapsto \begin{pmatrix}
        y & 0 \\
        0 & y^{-1}
        \end{pmatrix}$ and $\sigma_{2} :G_k \hookrightarrow SL_2(\C)$ given by $\omega \mapsto \begin{pmatrix}
        \omega & 0 \\
        0 & \omega^{-1}
        \end{pmatrix}$. Let $\chi_i(g):=\omega^i$ and denote $\chi=\chi_{1}$. We have that $\Oc(G_k)= \C \langle \chi \rangle \cong \C^{\Z_k}$. Moreover, 

\begin{multicols}{2}
\begin{align*}
     \Oc(SL_2(\C)) & \stackrel{\sigma_{1}^t}{\twoheadrightarrow} \Oc(G_m)=\C[X,X^{-1}] \qquad  \text{and}   \\
     X_{11} & \to X  \\
     X_{22} & \to X^{-1} \\
     X_{12} & \to 0 \\
     X_{21} & \to 0
\end{align*}
 
\begin{align*}
      \ \ \Oc(SL_2(\C)) & \stackrel{\sigma_{2}^t}{\twoheadrightarrow} \Oc(G_k) \cong \C \langle \chi | \chi^k=\varepsilon \rangle \cong \C^{\Z_k} \\
     X_{11} & \to \chi \\
     X_{22} & \to \chi^{-1} \\
     X_{12} & \to 0 \\
     X_{21} & \to 0
\end{align*}
\end{multicols}

We have the following commutative diagrams, finding different quotients of $\Oc_q(SL_2(\C))$ in each case.

    $$\xymatrix {
    \mathcal{O}(SL_2(\C)) \ar@{^{(}->}[r] \ar@{->>}[d] 
    & \mathcal{O}_q(SL_2(\C)) \ar@{->>}[r] \ar@{->>}[d] & \mathfrak{u}_q(\mathfrak{sl}_2)^{\circ} \ar@{->>}[d] \\
    \mathcal{O}(SL_2(\C))/(X_{12}, X_{21}) \ar@{^{(}->}[r] \ar@{->>}[d]^{\sigma_{1}^t}  & \mathcal{O}_q(SL_2(\C))/(b,c) \ar@{->>}[r] \ar@{->>}[d] & \mathcal{O}_q(SL_2(\C))/\langle a^{\ell}-1, d^{\ell}-1, b, c\rangle \ar@{=}[d] \\
    \Oc(G_{m}) \ar@{^{(}->}[r] & A_{\mathcal{D}} \ar@{->>}[r] & \C \Z_{\ell}
    }$$
where $A_{\mathcal{D}}=\C \langle a,b,c,d, X \ | \ b=c=0, ad=1, a^{\ell}=X \rangle$.
    $$\xymatrix {
    \mathcal{O}(SL_2(\C)) \ar@{^{(}->}[r] \ar@{->>}[d] 
    & \mathcal{O}_q(SL_2(\C)) \ar@{->>}[r] \ar@{->>}[d] & \mathfrak{u}_q(\mathfrak{sl}_2)^{\circ} \ar@{->>}[d] \\
    \mathcal{O}(SL_2(\C))/(X_{12}, X_{21}) \ar@{^{(}->}[r] \ar@{->>}[d]^{\sigma_{2}^t}  & \mathcal{O}_q(SL_2(\C))/(b, c) \ar@{->>}[r] \ar@{->>}[d] & \Oc_q(SL_2(\C))/\langle a^{\ell}-1, d^{\ell}-1, b, c\rangle \ar@{=}[d] \\
    \Oc(\Z_k) \ar@{^{(}->}[r] & A_{\mathcal{D}} \ar@{->>}[r] & \C \Z_{\ell}
    }$$
where $A_{\mathcal{D}}=\C \langle a,b,c,d, \chi \ | \ b=c=0, ad=1, a^{\ell}=\chi, \chi^k=1 \rangle$. 

Now  consider $\delta:\Z_t \to \widehat{\Z_k} \cong \langle \chi \rangle  $ such that $\delta(p)=\chi^r$. We have one more step.
    $$\xymatrix {
    \mathcal{O}(SL_2(\C)) \ar@{^{(}->}[r] \ar@{->>}[d] 
    & \mathcal{O}_q(SL_2(\C)) \ar@{->>}[r] \ar@{->>}[d] & \mathfrak{u}_q(\mathfrak{sl}_2)^{\circ} \ar@{->>}[d] \\
    \mathcal{O}(SL_2(\C))/(X_{12},X_{21}) \ar@{^{(}->}[r] \ar@{->>}[d]^{\sigma_{2}^t}  & \mathcal{O}_q(SL_2(\C))/(b,c) \ar@{->>}[r] \ar@{->>}[d] & \mathcal{O}_q(SL_2(\C))/\langle a^{\ell}-1, d^{\ell}-1, b, c\rangle \ar@{=}[d] \\
    \Oc(\Z_k) \ar@{^{(}->}[r] \ar@{=}[d]  & A_{\mathcal{D}} \ar@{->>}[r] \ar@{->>}[d] & \C\Z_{\ell} \ar@{->>}[d] \\
    \Oc(\Z_k) \ar@{^{(}->}[r]   & \overline{A_{\mathcal{D}}} \ar@{->>}[r]  & \C\Z_{\ell}/\C\Z_{t}}$$
\end{example}
\noindent where $\overline{A_{\mathcal{D}}}=\C \langle a,b,c,d, ,\chi \ | \ b=c=0, ad=1, a^{\ell}=\chi, \chi^k=1 ,\chi^r=a^{p} \text{ with } \ell=p t, rt\equiv 1 \mod k \rangle$. This last relation is due to 
\begin{center}
    $\chi^r = a^{p} \longrightarrow  \chi ^{rt}= a^{p t} = a^{\ell} = \chi  \longrightarrow rt \equiv 1 \mod k$.
\end{center}

\subsection{Determination of quantum subgroups.} \label{determinationodd}

In this section we prove that this construction is comprehensive: every quotient of Hopf algebras $\rho : \Oc_q(SL_2(\C)) \to A$ is determined by some odd subgroup data. We will follow \cite[Section 3]{AG}.

Given $\rho : \Oc_q(SL_2(\C)) \to A$ an epimorphism, we know that $\Oc_q(SL_2(\C))$ fits into an exact sequence. By  Proposition \ref{prop1mu}, $A$ also fits into an exact sequence
 $K  \hookrightarrow A \twoheadrightarrow H$.

As $K$ is commutative, there exists a group $\Gamma$ and a injective map $\sigma:\Gamma \to SL_2(\C)$ such that $K \cong \Oc(\Gamma)$. Also, by \cite[(1.12)]{AG} there exists an epimorphism $r: \mathfrak{u}_q(\mathfrak{sl}_2)^{\circ} \to H$ and $H ^{\circ}$ is determined by a triple $(\Sigma, I_{+},I_{-})$, where $\C \Sigma \subseteq H^{\circ} \subseteq \mathfrak{u}_q(\mathfrak{sl}_2)$. We have the following commutative diagram

$$\xymatrix {
    \mathcal{O}(SL_2(\C)) \ar@{^{(}->}[r]^{\iota} \ar@{->>}[d]^{\sigma^t} 
    & \mathcal{O}_q(SL_2(\C)) \ar@{->>}[r]^{\pi} \ar@{->>}[d]^{\rho} & \mathfrak{u}_q(\mathfrak{sl}_2)^{\circ} \ar@{->>}[d]^{r} \\
    K \ar@{^{(}->}[r]^{\bar{\iota}}  & A \ar@{->>}[r]^{\bar{\pi}} & H
    }$$

 We want to see that $A=A_{\mathcal{D}}$ for some odd data subgroup $(I_{+}, I_{-}, N, \Gamma, \sigma, \delta)$. Consider the subalgebra $\mathfrak{u}_q(\mathfrak{l})$ where $\mathfrak{l} = \mathfrak{l}_{+} \oplus \C h \oplus \mathfrak{l}_{-}$ as in Remark \ref{defl}. As $\Sigma \subseteq T$, by \cite[Lemma 3.1]{AG} we have that $H^{\circ} \subseteq \mathfrak{u}_q(\mathfrak{l}) $  so $w:\mathfrak{u}_q(\mathfrak{l})^{\circ} \to H$ is an epimorphism and we have the commutative diagram

$$\xymatrix{
    \mathcal{O}(SL_2(\C)) \ar@{^{(}->}[r]^{\iota} \ar@{->>}[d]^{u_L} 
    & \mathcal{O}_q(SL_2(\C)) \ar@{->>}[r]^{\pi} \ar@{->>}[d]^{v_L} & 
    \mathfrak{u}_q(\mathfrak{sl}_2)^{\circ} \ar@{->>}[d]^{w_L} \\
    \mathcal{O}(L) \ar@{^{(}->}[r]^{\iota_L} \ar@{->>}[d]^{u}  & \mathcal{O}_q(L) \ar@{->>}[r]^{\pi_L} \ar@{->>}[d]^{v} & \mathfrak{u}_q(\mathfrak{l})^{\circ} \ar@{->>}[d]^{w} \\
    \Oc(\Gamma) \ar@{^{(}->}[r]^{\widehat{\iota}} & A \ar@{->>}[r]^{\widehat{\pi}} & H
    }$$
    
Moreover by \cite[Lemma 3.2]{AG}, $\sigma(\Gamma) \subseteq L$ and $A$ is a quotient of $A_{q,\mathfrak{l},\sigma}$ given by the pushout. Thus, we have the commutative diagram 

$$\xymatrix {
    \mathcal{O}(SL_2(\C)) \ar@{^{(}->}[r]^{\iota} \ar@{->>}[d]^{u_L} 
    & \mathcal{O}_q(SL_2(\C)) \ar@{->>}[r]^{\pi} \ar@{->>}[d]^{v_L} & \mathfrak{u}_q(\mathfrak{sl}_2)^{\circ} \ar@{->>}[d]^{w_L} \\
    \mathcal{O}(L) \ar@{^{(}->}[r]^{\iota_L} \ar@{->>}[d]^{u}  & \mathcal{O}_q(L) \ar@{->>}[r]^{\pi_L} \ar@{->>}[d]^{s} & \mathfrak{u}_q(\mathfrak{l})^{\circ} \ar@{=}[d] \\
    \Oc(\Gamma) \ar@{^{(}->}[r]^{\widehat{\iota}} \ar@{=}[d] & A_{q,\mathfrak{l},\sigma} \ar@{->>}[r]^{\widehat{\pi}} \ar@{->>}[d]^{t} & \mathfrak{u}_q(\mathfrak{l})^{\circ} \ar@{->>}[d]^{w} \\
    \Oc(\Gamma) \ar@{^{(}->}[r]^{\bar{\iota}} & A \ar@{->>}[r]^{\bar{\pi}} & H
    }$$
In the case $H^{\circ}= \mathfrak{u}_{q}(\mathfrak{sl}_2), H^{\circ}= \mathfrak{u}_{q}(\mathfrak{b}^{+}), H^{\circ}= \mathfrak{u}_{q}(\mathfrak{b}^{-})$ or $H^{\circ}= \mathfrak{u}_{q}(\mathfrak{h})$, the map $w$ is an isomorphism. If $H^{\circ}\subset \mathfrak{u}_{q}(\mathfrak{h})$, let $\C N \cong \ker w$ with $N$ a subgroup of $\Z_{\ell}$, so $N=(p) = \Z_t$ for some $t$. By \cite[Lemma 3.3]{AG}, there exists $\delta:N \to \widehat{\Gamma}$ such that $A \cong A_{q,\mathfrak{l},\sigma}/J_{\delta}$.

\subsection{Quantum subgroups of $\Oc_{-1}(SL_{2}(\C))$} \label{section-1}

Recall that  $\Oc_{-1}(SL_2(\C))$ fits in the exact sequence

        \begin{center}
        $\mathcal{O}(PSL_2(\C)) \hookrightarrow \mathcal{O}_{-1}(SL_2(\C)) \twoheadrightarrow \C \Z_2 $.
        \end{center}






In this case, the quotients of $\Oc_{-1}(SL_2(\C))$ are determined by the subgroups of $PSL_2(\C)$, as we see in the following result.

\begin{thm} \label{teo-1}

Let $A$ be a  Hopf algebra quotient of $\Oc_{-1}(SL_2(\C))$. Then either
    \begin{enumerate}
        \item [(i)] $A$ is isomorphic to a quotient  
        $\mathcal{O}_{-1}(SL_2(\C))/(\mathcal{J})$ where $\mathcal{J}=\ker\sigma^t$ and $\sigma: \Gamma \hookrightarrow PSL_2(\C)$ is a monomorphism of algebraic groups, or
        \item [(ii)] $A$ is isomorphic to 
        a function algebra over some dihedral group $D_{2m}$.
        \end{enumerate}


    
\end{thm}

\begin{proof} 
    The two possibilities for the quotients are related to the possible subgroups of $\C\Z_2$.
    Let $\rho: \Oc_{-1}(SL_2(\C)) \to A$ be a Hopf algebra quotient. By Proposition \ref{prop1mu}, $A$ is contained in an exact sequence of Hopf algebras $B  \hookrightarrow A \twoheadrightarrow H$, and the following diagram is commutative

$$\xymatrix {
    \mathcal{O}(PSL_2(\C)) \ar@{^{(}->}[r]^{\iota} \ar@{->>}[d]^{p} 
    & \mathcal{O}_{-1}(SL_2(\C)) \ar@{->>}[r]^(.6){\pi} \ar@{->>}[d]^{\rho} & \C\Z_2 \ar[d]^{r} \\
    B \ar@{^{(}->}[r]^{\bar{\iota}}  & A \ar@{->>}[r]^{\bar{\pi}} & H 
    }$$ 

 \noindent   where $H$ is a quotient of $\C \Z_2$ and $B \cong \Oc(\Gamma)$ for some $\Gamma $ algebraic subgroup of $ PSL_2(\C)$. Write $\sigma: \Gamma \hookrightarrow PSL_2(\C)$ for the monomorphism of algebraic groups. Then $\sigma^t: \Oc(PSL_2(\C)) \twoheadrightarrow \Oc(\Gamma)$ is an epimorphism of Hopf algebras. 
    
    If $H=\C\Z_2$, by Proposition \ref{prop2mu}, $\Oc_{-1}(SL_2(\C))/(\ker \sigma^t)$ is contained in an exact sequence of Hopf algebras

$$\Oc(\Gamma)  \hookrightarrow \Oc_{-1}(SL_2(\C))/(\ker \sigma^t) \twoheadrightarrow \C \Z_2/\pi(\ker \sigma^t) \cong \C \Z_2 $$

and the following diagram is commutative

$$\xymatrix {
    \mathcal{O}(PSL_2(\C)) \ar@{^{(}->}[r]^{\iota} \ar@{->>}[d]^{\sigma^t} 
    & \mathcal{O}_{-1}(SL_2(\C)) \ar@{->>}[r]^{\pi} \ar@{->>}[d] & \C\Z_2 \ar@{=}[d] \\
    \mathcal{O}(\Gamma) \ar@{^{(}->}[r]^(0.3){\bar{\iota}}  & \mathcal{O}_{-1}(SL_2(\C))/(\text{ker } \sigma^t) \ar@{->>}[r]^(0.7){\bar{\pi}} & \C \Z_2 
    }$$

We want to prove that $A \cong \mathcal{O}_{-1}(SL_2(\C))/(\text{ker } \sigma^t) $. Since $\mathcal{O}_{-1}(SL_2(\C))/(\text{ker } \sigma^t)$ is a pushout, then there exists $\psi: \mathcal{O}_{-1}(SL_2(\C))/(\text{ker } \sigma^t) \to A$ an epimorphism of Hopf algebras and we have the commutative diagram

$$\xymatrix {
    \mathcal{O}(\Gamma) \ar@{^{(}->}[r]^(0.3){\bar{\iota}} \ar@{=}[d] 
    & \mathcal{O}_{-1}(SL_2(\C))/(\text{ker } \sigma^t) \ar@{->>}[r]^(0.7){\bar{\pi}} \ar@{->>}[d]^{\psi} & \C\Z_2 \ar@{=}[d] \\
    \Oc(\Gamma) \ar@{^{(}->}[r]^{\widehat{\iota}}  & A \ar@{->>}[r]^{\widehat{\pi}} & \C \Z_2 
    }$$ 

Hence, by Proposition \ref{propiso}, $\psi$ is an isomorphism.

\bigbreak
If $H=\C$, we have that the epimorphism of Hopf algebras $r:\C \Z_2 \to \C$
is given by the counit $\varepsilon:\C \Z_2 \to \C$  and the following diagram is commutative

$$\xymatrix {
    \mathcal{O}(PSL_2(\C)) \ar@{^{(}->}[r]^{\iota} \ar@{->>}[d]^{\sigma^t} 
    & \mathcal{O}_{-1}(SL_2(\C)) \ar@{->>}[r]^{\pi} \ar@{->>}[d]^{\rho} & \C \Z_2 \ar@{->>}[d]^{\varepsilon} \\
    \Oc(\Gamma) \ar@{^{(}->}[r]^{\widehat{\iota}} & A \ar@{->>}[r]^{\widehat{\pi}} & \C
    }$$
 Thus, by Remark \ref{prop4mu} it follows that $A \cong \Oc(\Gamma)$. Let us see that $\Gamma \cong D_{2m}$ for some dihedral group $D_{2m}$.

 As $A \cong \Oc(\Gamma)$, then $A$ is commutative. If we denote by $\bar{a}, \bar{b}, \bar{c}, \bar{d}$ the elements in $A$, we have that $\bar{a} \bar{b} = \bar{b}\bar{a}$ and $\bar{a} \bar{b} = - \bar{b}\bar{a}$ so $\bar{a} \bar{b} = 0$. In the same way, $\bar{a} \bar{c} = \bar{b} \bar{d} = \bar{c} \bar{d} = 0$. 
 In order to determine $\Gamma$, we analyze the image of the spectrum functor $\Alg(A,\C)$.  
 For any $\alpha: A \to \C$ algebra morphism, we have

\begin{center} 
$\alpha(\bar{a}\bar{c})=\alpha(\bar{a})\alpha(\bar{c})=0$ \\
$\alpha(\bar{a}\bar{b})=\alpha(\bar{a})\alpha(\bar{b})=0$ \\
$\alpha(\bar{a}\bar{d})=\alpha(\bar{a})\alpha(\bar{d})$ \\
$\alpha(\bar{a}\bar{d} + \bar{b}\bar{c})= \alpha(\bar{a}\bar{d}) + \alpha (\bar{b}\bar{c}) =1 $ 
\end{center}

So, if $\alpha(\bar{a}) \neq0$ then $\alpha(\bar{b})=\alpha(\bar{c})=0$ and $\alpha(\bar{d})=\alpha(\bar{a})^{-1}$. And if $\alpha(\bar{a}) =0$ then $\alpha(\bar{d})=0$ and $\alpha(\bar{c})=\alpha(\bar{b})^{-1}$. These conditions determine 
uniquely two algebra maps 
$\alpha,\beta \in \Alg(A,\C)$

\begin{center}
    $\alpha: \bar{a} \to B  \hspace{1.8cm} \beta: \bar{a} \to 0$ \\
$\quad \ \ \bar{b} \to 0  \hspace{1.8cm} \quad \ \ \  \bar{b} \to C$ \\
$\quad \quad \ \ \bar{c} \to 0  \hspace{1.8cm} \quad \ \ \  \bar{c} \to C^{-1}$ \\
$\quad  \ \bar{d} \to B^{-1}  \hspace{1.8cm}   \   \bar{d} \to 0$ \\
\end{center}

\noindent for some $B,C \in \C^{\times}$.

Let $\langle\alpha, \beta \rangle$ be the group generated by these maps. By the computation above,
it follows that $\langle \alpha, \beta \rangle \simeq \Alg(A,\C)=\Gamma$. Moreover, since  
$\beta$ is an involution, then $\Gamma$ is a quotient of the  semidirect product $\C^{*} \rtimes \Z_2$ and  

\begin{center}
    $\C^{*} \rtimes \Z_2 \twoheadrightarrow \Gamma \hookrightarrow PSL_2(\C)$;
\end{center}
here we are considering the action 
of $\beta$ on $\alpha$ given by
the conjugation $\beta\cdot \alpha=\beta\alpha\beta^{-1}$. 
Since the only subgroups of $PSL_2(\C)$ that are quotients of $\C^{*} \rtimes \Z_2$
are the (finite) dihedral groups, it follows that $\Gamma \cong D_{2m}$ for some $m\in \N$.



\end{proof}

\begin{remark}
    It follows directly that the quotients in (i) are quantum subgroups. Also, given a dihedral group $D_{2m}$ generated by an element of order $m$ and an element of order $2$, we can define the morphisms $\alpha, \beta$ just taking $B$ as the generator of order $m$ and $C$ as the generator of order $2$. This means that for any dihedral group $D_{2m}$ we have the Hopf algebra quotient of $\Oc_{-1}(SL_2(\C))\twoheadrightarrow \Oc(D_{2m}) =\C^{D_{2m}}$.
\end{remark}

To give an explicit definition for the quotients in $(i)$ above we have to compute $\text{ker } \sigma^t$. We show how to do this in the following example.

\begin{example} \label{zn} 

Consider $\tilde{\Gamma}$ a subgroup of $SL_2(\C)$ and let $p: \tilde{\Gamma} \twoheadrightarrow \Gamma, \pi: SL_2(\C) \twoheadrightarrow PSL_2(\C)$ be the corresponding epimorphisms. So we have

    $$\xymatrix {
    \Oc(\tilde{\Gamma})   
    & \Oc(SL_2(\C)) \ar@{->>}[l]^{i^t}  \\
    \Oc(\Gamma) \ar@{^{(}->}[u]^{p^{*}}   & \Oc(PSL_2(\C)) \ar@{->>}[l]^{\sigma^t} \ar@{^{(}->}[u]^{\pi^{*}} 
    }$$

Let $\Gamma=\Z_n= \langle g \rangle$, $\tilde{\Gamma}=\Z_{2n}= \langle h \rangle$ and $\omega$ a root of unity of order $2n$. Consider the monomorphism $\iota: \Z_{2n} \hookrightarrow SL_2(\C)$ given by $h \mapsto \begin{pmatrix}
        \omega & 0 \\
        0 & \omega^{-1}
        \end{pmatrix}$ and the following diagram
    $$\xymatrix {
    \Z_{2n} \ar@{^{(}->}[r]^{\iota} \ar@{->>}[d]^{P_\Z} 
    & SL_2(\C) \ar@{->>}[d]^{\pi} \\
    \Z_n \ar@{^{(}->}[r]^{\sigma}  & PSL_2(\C) 
    }$$

where $P_{\Z}= \pi \circ \iota$ and $g=P_{\Z}(h)$; hence, $Ker(P_\Z)=\langle h^n\rangle.$ Dualizing we obtain   

    $$\xymatrix {
    \C^{\Z_{2n}}   
    & \Oc(SL_2(\C)) \ar@{->>}[l]^{\iota^{t}}  \\
    \C^{\Z_n} \ar@{^{(}->}[u]^{P_{\Z}^{t}}   & \Oc(PSL_2(\C)) \ar@{->>}[l]^{\sigma^t} \ar@{^{(}->}[u]^{\pi^{t}} 
    }$$

The commutative diagram giving the corresponding quantum subgroup is

$$\xymatrix {
    \mathcal{O}(PSL_2(\C)) \ar@{^{(}->}[r]^{\iota} \ar@{->>}[d]^{\sigma^{t}} 
    & \mathcal{O}_{-1}(SL_2(\C)) \ar@{->>}[r]^{\pi} \ar@{->>}[d]^{\rho} & \C\Z_2 \ar@{=}[d] \\
    \C^{\Z_n} \ar@{^{(}->}[r]^(0.3){\bar{\iota}}  & \mathcal{O}_{-1}(SL_2(\C))/(\ker \sigma ^{t}) \ar@{->>}[r]^(0.7){\bar{\pi}} & \C\Z_2
    }$$

We want to compute $\ker \sigma^t $. Suppose $\C^{\Z_n}=\langle\alpha\rangle$ where $\alpha(g)=\omega^2$, and $C^{\Z_{2n}}=\langle\chi\rangle$ where $\chi(h)=\omega$. Then $\sigma^t(X_{11}^2)(g)=X_{11}^2(\sigma(g))=\omega^2 = \alpha(g)$ and this implies $\sigma^t(X_{11}^2)=\alpha$. By the same argument, $\sigma^t(X_{22}^2)=\alpha^{-1}$, $\sigma^t(X_{12} \ \cdot)=\sigma^t(X_{21} \ \cdot)=\sigma^t(\cdot \ X_{12})=\sigma^t(\cdot \ X_{21})=0$, and $\sigma^t(X_{11}X_{22})=\varepsilon$. 

Observe that $\sigma^t(X_{11}^{2n})(g)=\alpha(g)^n=\omega^{2n}=1$ and $\sigma^t(X_{22}^{2n})(g)=\alpha^{-1}(g)^n=\omega^{-2n}=1$. So,  $\ker \sigma^t= \langle X_{11}^{2n}-1, X_{22}^{2n}-1, X_{11}X_{22}-1,  X_{ij}X_{kl}  \text{ with } i \neq j \text{ or } k \neq l \rangle$. 

However, as $ad+bc=1$ we have that $b(ad+bc)=b$ so  \\ $\rho(b)=\rho(bad) + \rho(bbc)= 0$. By the same computation, $\rho(c)=0$. This implies that $b,c \in (\text{ker } \sigma^{t})$.

Then, $\mathcal{O}_{-1}(SL_2(\C))/(\text{ker } \sigma^{t}) = \C \langle a | a^{2n}=1 \rangle \cong \C ^{\Z_{2n}} \cong \C \Z_{2n} $ and the commutative diagram above reads 

$$\xymatrix {
    \mathcal{O}(PSL_2(\C)) \ar@{^{(}->}[r]^{\iota} \ar@{->>}[d]^{\sigma^t} 
    & \mathcal{O}_{-1}(SL_2(\C)) \ar@{->>}[r]^(0.6){\pi} \ar@{->>}[d]^{\rho} & \C\Z_2 \ar@{=}[d] \\
    \C^{\Z_n} \ar@{^{(}->}[r]^{\bar{\iota}}  & \C^{\Z_{2n}} \ar@{->>}[r]^{\bar{\pi}} & \C\Z_2.
    }$$
\end{example}

To finish this section, we study when two quotients of type (i) are isomorphic. Recall Definition \ref{isoquot} for quantum sobgroup isomorphism.


\begin{prop}
Let $\sigma_i: \Gamma_i \to PSL_2(\C)$ be monomorphism of algebraic groups for $i=1,2$. Then  
$\mathcal{O}_{-1}(SL_2(\C))/(\ker \sigma_1^t)$ and $\mathcal{O}_{-1}(SL_2(\C))/(\ker \sigma_2^t)$ are isomorphic as quantum subgroups if and only if there exists $\rho:\Gamma_2 \to \Gamma_1$ isomorphism of algebraic groups such that $ \sigma_1 \circ \rho=\sigma_2$.
\end{prop}

\begin{proof}
    Denote $A_1=\mathcal{O}_{-1}(SL_2(\C))/(\mathcal{J}_1)$ and $A_2=\mathcal{O}_{-1}(SL_2(\C))/(\mathcal{J}_2)$ where $\mathcal{J}_1=\ker\sigma_1^t$ and $\mathcal{J}_2=\ker\sigma_2^t$. Suppose $\varphi: A_1 \to A_2$ is an isomorphism, so $\varphi \circ \rho_1=\rho_2$. We have the following commutative diagram

\begin{equation*} 
\xymatrix{
    \Oc(PSL_2(\C)) \ar@{^{(}->}[r]^{\iota} \ar@{->>}[d]^{\sigma_1^t} \ar@/_3.6pc/[dd]_{\sigma_2^t} 
    & \Oc_{-1}(SL_2(\C)) \ar@/^1.9pc/[dd]^(0.65){\rho_2} \ar@{->>}[r]^{\pi} \ar@{->>}[d]^{\rho_1}  & 
   \C \Z_2 \ar@{=}[d]  \\
    \Oc(PSL_2(\C))/\mathcal{J}_1 \ar@{^{(}->}[r]^(0.6){\iota_1} \ar@{-->}[d]^{\sigma^t}  & A_1 \ar@{->>}[r]^{\pi_1} \ar@{-->}[d]^{\varphi} & \C \Z_2 \ar@{=}[d] \\
    \Oc(PSL_2(\C))/\mathcal{J}_2 \ar@{^{(}->}[r]^(0.6){\iota_2} & A_2 \ar@{->>}[r]^{\pi_2} & \C \Z_2
    }
\end{equation*}
We have  $\mathcal{J}_1 \subseteq \ker \sigma_2^t= \mathcal{J}_2$ since $0=\varphi \iota_1\sigma_1^t(\mathcal{J}_1)= \varphi \rho_1\iota(\mathcal{J}_1) = \rho_2 \iota(\mathcal{J}_1)= \iota_2 \sigma_2^t(\mathcal{J}_1)$. Using $\varphi^{-1}$ we can see that $\mathcal{J}_2 \subseteq \mathcal{J}_1$ so $\mathcal{J}_1=\mathcal{J}_2$.

As $\mathcal{J}_1= \ker \sigma_1^t=\{ p \in \Oc(PSL_2(\C)) : p|_{\sigma_1(\Gamma_1)}=0\}$ and $\mathcal{J}_2= \ker \sigma_2^t=\{ p \in \Oc(PSL_2(\C)) : p|_{\sigma_2(\Gamma_2)}=0\}$, $\mathcal{J}_1=\mathcal{J}_2$ means that $\sigma_1(\Gamma_1)= \sigma_2(\Gamma_2)$, thus we can define $\rho: \Gamma_2 \to \Gamma_1$ such that  $\rho(\Gamma_2)=\sigma_1^{-1}( \sigma_2(\Gamma_2))= \sigma_1^{-1}(\sigma_1(\Gamma_1))= \Gamma_1$. So $\sigma_1 \rho(\Gamma_2)=\sigma_2(\Gamma_2)$ and $\sigma_1 \circ \rho = \sigma_2$. \\

On the other hand, suppose there exists an isomorphism $\rho: \Gamma_2 \to \Gamma_1$ such that $\sigma_1 \circ \rho = \sigma_2$. Then the following diagram is commutative.

$$\xymatrix {
    \Oc(\Gamma_2) 
    & \Oc(PSL_2(\C)) \ar@{->>}[l]^{\sigma_2^t}  \ar@{->>}[dl]^{\sigma_1^t} \\
    \Oc(\Gamma_1) \ar@{->}[u]^{\rho^t}
    }$$

That is, $\sigma_2^t=\rho^t \circ \sigma_1^t$. Then $\mathcal{J}_1= \ker \sigma_1^t \subseteq \ker \sigma_2^t = \mathcal{J}_2$. Using $\rho^{-1}$ we can see that $\mathcal{J}_2 \subseteq \mathcal{J}_1$ and so $\mathcal{J}_1=\mathcal{J}_2$.

Now, let us define $\varphi: \Oc_{-1}(SL_2)/(\mathcal{J}_1) \to \Oc_{-1}(SL_2)/(\mathcal{J}_2)$ such that $\varphi \circ \rho_1=\rho_2$. For $\bar{x} \in \Oc_{-1}(SL_2)/(\mathcal{J}_1)$, $\bar{x}=\rho_1(x)$ with $x \in \Oc_{-1}(SL_2)$, define $\varphi(\bar{x})=\rho_2(x)$. Then $\varphi$ is well defined since, if $x-y \in (\mathcal{J}_1)=(\mathcal{J}_2)$, then $\rho_2(x-y)=0$, that is, $\rho_2(x)=\rho_2(y)$. Finally, $\varphi \circ \rho_1=\rho_2$ by definition.

\end{proof}

\begin{remark}
    In \cite[Theorem 5.19]{Bichon}, Bichon and Natale found all the noncommutative finite-dimensional Hopf algebra quotients of $\Oc_{-1}(SL_2(\C))$, that are in particular semisimple and cosemisimple. Here we extend this result by adding the infinite-dimensional ones.
\end{remark}
 
\subsection{Quantum subgroups of $\Oc_{q}(SL_{2}(\C))$, ord($q$) even} \label{sectioneven} As we have seen in Theorem \ref{teoremase}, if ord($q)=\ell=2m$ ($m \neq 1$), then $\Oc_q(SL_2(\C))$ fits in the exact sequence

    \bigskip
    
    \begin{center}
    $\mathcal{O}(PSL_2(\C)) \hookrightarrow \mathcal{O}_q(SL_2(\C)) \twoheadrightarrow \mathfrak{u}_{2,q}(\mathfrak{sl}_2)^{\circ} $.
    \end{center}

This case has been studied in \cite{Bichon} and \cite{Chelsea}, for finite dimensional quantum subgroups. Here we present all the quantum subgroups, and show an example in which we can determine the quantum subgroup explicitly via generators and relations.

\begin{defi}\label{def:subgroup-data-qeven}
    Let $q$ be a root of unity of even order $\ell\neq 2$. An even subgroup data is a collection $\mathcal{D} = (I_{+}, I_{-}, N, \Gamma, \sigma, \delta)$ such that:

\begin{itemize}
    \item $I_{+}, I_{-} \in \{\emptyset, \{1\}\}$. 
    
    \item $N$ is a subgroup of $(\mathbb{Z}_{\ell})^{s}$, where $s = 1 - |I_{+} \cup I_{-}|$.
    
    \item $\Gamma$ is an algebraic group.
    
    \item $\sigma \colon \Gamma \to \bar{L}$ is an injective group morphism, where $\bar{L} \subseteq PSL_2(\C)$ is determined by the sets $I_{+}$ and $I_{-}$. 
    
    \item $\delta \colon N \to \widehat{\Gamma}$ is a group morphism.
\end{itemize}
\end{defi}

\begin{remark}
   As in the odd case, $I_{+}, I_{-} \in \{\emptyset, \{1\}\}$ where the choice $I_{+}=\{\emptyset\}$ corresponds to taking the quotient by the ideal generated by $c$, $I_{-}=\{\emptyset\}$ corresponds to taking the ideal generated by $b$ and if $I_{\pm}=\{1\}$ we do not take quotients. These sets define an algebraic Lie subalgebra $\mathfrak{l}$ of $\mathfrak{sl}_{2}(\C)$ with connected Lie subgroup $L$ of $SL_2(\C)$ and a subgroup $\bar{L}$ of $PSL_2(\C)$, such that $\mathfrak{l} = \mathfrak{l}_{+} \oplus \C h \oplus \mathfrak{l}_{-}$ and $\mathfrak{l}_{+} = \C e$ or $\mathfrak{l}_{+} = 0$ and $\mathfrak{l}_{-} = \C f$ or $\mathfrak{l}_{-} = 0$.
\end{remark}

We have the following result that characterizes all the subgroups of $\Oc_q(SL_2(\C))$ for $q$ a root of unity of even order $\ell \neq 2$. The notion of isomorphism between two quotients of $\Oc_q(SL_2(\C))$ and equivalence between two subgroup data are given in Definition \ref{isoquot} and Definition \ref{isodata}, respectively.

\begin{thm} \label{teorema}
    Let $q$ be a root of unity of even order $\ell =2m , m \neq 1$. There exists a bijection between
    \begin{enumerate}
        \item[(i)]  Hopf algebra quotients $\rho:\Oc_q(SL_2(\C)) \to A$ up to isomorphism.
        \item[(ii)] Even subgroup datum up to equivalence.
    \end{enumerate}
\end{thm}

We prove the theorem in Sections \ref{constructioneven} and \ref{determinationeven}. First we prove, in Lemma \ref{lemaconst} that a subgroup datum $\mathcal{D} = (I_{+}, I_{-}, N, \Gamma, \sigma, \delta)$ gives a quotient of $\Oc_q(SL_2(\C))$, and then, in Lemma \ref{lemadet}, we show that every quotient is determined by some even subgroup datum. The proof that it is a bijection follows mutatis mutandis from \cite[ Theorem 2.20]{AG}.

\subsection{Construction of quantum subgroups.} \label{constructioneven}

\begin{lemma} \label{lemaconst}
For every even subgroup data $\mathcal{D} = (I_{+}, I_{-}, N, \Gamma, \sigma, \delta)$, there exists a Hopf algebra $A_{\mathcal{D}}$ that is a quotient of $\Oc_{q}(SL_2(\C))$ and fits into the exact sequence
    $$ 1 \rightarrow \Oc(\Gamma) \xrightarrow{\bar{\iota}} A_{\mathcal{D}} \xrightarrow{\bar{\pi}} H \rightarrow 1, $$

Moreover, $A_{\mathcal{D}}$ is given by the quotient of  $\mathcal{O}_q(L)/(\mathcal{J})$ by $J_{\delta}$ where $\mathcal{J}=\ker \sigma^t$ and  $J_{\delta}$ is the two-sided ideal generated by some grouplike elements and the following diagram of exact sequences of Hopf algebras is commutative:

$$\xymatrix {
    \mathcal{O}(PSL_2(\C)) \ar@{^{(}->}[r]^{\iota} \ar@{->>}[d]^{u_L} 
    & \mathcal{O}_q(SL_2(\C)) \ar@{->>}[r]^{\pi} \ar@{->>}[d]^{v_L} & \mathfrak{u}_{2,q}(\mathfrak{sl}_2)^{\circ} \ar@{->>}[d]^{w_L} \\
    \mathcal{O}(\bar{L}) \ar@{^{(}->}[r]^{\iota_L} \ar@{->>}[d]^{\sigma^t}  & \mathcal{O}_q(L) \ar@{->>}[r]^{\pi_L} \ar@{->>}[d]^{s} & \mathfrak{u}_{2,q}(\mathfrak{l})^{\circ} \ar@{=}[d] \\
    \Oc(\Gamma) \ar@{^{(}->}[r]^{\widehat{\iota}} \ar@{=}[d] & \mathcal{O}_q(L)/(\mathcal{J}) \ar@{->>}[r]^{\widehat{\pi}} \ar@{->>}[d]^{t} & \mathfrak{u}_{2,q}(\mathfrak{l})^{\circ} \ar@{->>}[d]^{w} \\
    \Oc(\Gamma) \ar@{^{(}->}[r]^{\bar{\iota}} & A_{\mathcal{D}} \ar@{->>}[r]^{\bar{\pi}} & H
    }$$
\end{lemma}

As in the odd case, the construction is done in three steps where the first one is as follows: Take the quotient by the ideals $(b), (c)$ or $(b,c)$, this corresponds to the choice of $I_{+}, I_{-}$. Then we can construct a Hopf subalgebra $\mathfrak{u}_{2,q}(\mathfrak{l}) \subseteq \mathfrak{u}_{2,q}(\mathfrak{sl}_2)$ and algebraic subgroups $L \subseteq SL_2(\C)$ and $\bar{L} \subseteq PSL_2(\C)$ such that $Lie(L)=\mathfrak{l}$. We have an epimorphism  $u_L : \mathfrak{u}_{2,q}(\mathfrak{sl}_2)^{\circ}
\rightarrow \mathfrak{u}_{2,q}(\mathfrak{l})^{\circ}$ of Hopf algebras and
 the following commutative
diagram of exact sequences of Hopf algebras 

$$\xymatrix {
    \mathcal{O}(PSL_2(\C)) \ar@{^{(}->}[r]^{\iota} \ar@{->>}[d]^{u_L} 
    & \mathcal{O}_q(SL_2(\C)) \ar@{->>}[r]^{\pi} \ar@{->>}[d]^{v_L} & \mathfrak{u}_{2,q}(\mathfrak{sl}_2)^{\circ} \ar@{->>}[d]^{w_L} \\
    \mathcal{O}(\bar{L}) \ar@{^{(}->}[r]^{\iota_L}  & \mathcal{O}_q(L) \ar@{->>}[r]^{\pi_L} & \mathfrak{u}_{2,q}(\mathfrak{l})^{\circ}
    }$$

After this construction we obtain four possible diagrams that are shown below.

\begin{enumerate}
    \item[Case I:] Take the quotient by the ideals generated by $b$ and $c$ so $\mathfrak{l} =  \C h$ and $\mathfrak{u}_{2,q}(\mathfrak{h})$ is the subalgebra of $\mathfrak{u}_{2,q}(\mathfrak{sl}_2)$ generated by $k$. The diagram in this case is 
    $$\xymatrix {
    \mathcal{O}(PSL_2(\C)) \ar@{^{(}->}[r]^{\iota} \ar@{->>}[d]^{u_L} 
    & \mathcal{O}_q(SL_2(\C)) \ar@{->>}[r]^{\pi} \ar@{->>}[d]^{v_L} & \mathfrak{u}_{2,q}(\mathfrak{sl}_2)^{\circ} \ar@{->>}[d]^{w_L} \\
    \Oc(\bar{T})  \ar@{^{(}->}[r]^(.4){\iota_L}  & \mathcal{O}_q(SL_2(\C))/(b,c) \ar@{->>}[r]^{\pi_L} & \mathfrak{u}_{2,q}(\mathfrak{h})^{\circ} \simeq \C\Z_{\ell}
    }$$

\noindent Here $ \Oc(\bar{T}) \simeq \mathcal{O}(PSL_2(\C))/(X_{12}^2, X_{21}^2,X_{11}X_{12},X_{11}X_{21},X_{12}X_{21},X_{12}X_{22},X_{21}X_{22})$.

\item[Case II:] Take the quotient by the ideal $\{c\}$ so $\mathfrak{l} = \C e \oplus \C h$ and $\mathfrak{u}_{2,q}(\mathfrak{b}^{+})$ is the subalgebra of $\mathfrak{u}_{2,q}(\mathfrak{sl}_2)$ generated by $k$ and $e$. The diagram in this case is 
    $$\xymatrix {
    \mathcal{O}(PSL_2(\C)) \ar@{^{(}->}[r]^{\iota} \ar@{->>}[d]^{u_L} 
    & \mathcal{O}_q(SL_2(\C)) \ar@{->>}[r]^{\pi} \ar@{->>}[d]^{v_L} & \mathfrak{u}_{2,q}(\mathfrak{sl}_2)^{\circ} \ar@{->>}[d]^{w_L} \\
    \Oc(\bar{B}^{+})  \ar@{^{(}->}[r]^(.4){\iota_L}  & \mathcal{O}_q(SL_2(\C))/(c) \ar@{->>}[r]^(.6){\pi_L} & \mathfrak{u}_{2,q}(\mathfrak{b}^{+})^{\circ} 
    }$$

\noindent Here $ \Oc(\bar{B}^{+}) \simeq \mathcal{O}(PSL_2(\C))/(X_{11}X_{21}, X_{12}X_{21}, X_{21}^2, X_{21}X_{22})$.

\item[Case III:] Take the quotient by the ideal $\{b\}$ so $\mathfrak{l} = \C h \oplus \C f$ and $\mathfrak{u}_{2,q}(\mathfrak{b}^{-})$ is the subalgebra of $\mathfrak{u}_{2,q}(\mathfrak{sl}_2)$ generated by $k$ and $f$. The diagram in this case is 
    $$\xymatrix {
    \mathcal{O}(PSL_2(\C)) \ar@{^{(}->}[r]^{\iota} \ar@{->>}[d]^{u_L} 
    & \mathcal{O}_q(SL_2(\C)) \ar@{->>}[r]^{\pi} \ar@{->>}[d]^{v_L} & \mathfrak{u}_{2,q}(\mathfrak{sl}_2)^{\circ} \ar@{->>}[d]^{w_L} \\
    \Oc(\bar{B}^{-})  \ar@{^{(}->}[r]^(.4){\iota_L}  & \mathcal{O}_q(SL_2(\C))/(b) \ar@{->>}[r]^(.6){\pi_L} & \mathfrak{u}_{2,q}(\mathfrak{b}^{-})^{\circ} 
    }$$

\noindent Here $ \Oc(\bar{B}^{-}) \simeq \mathcal{O}(PSL_2(\C))/(X_{12}^2, X_{11}X_{12},X_{12}X_{21}, X_{12},X_{22}) $.

    \item[Case IV:] We do not take quotients so $\mathfrak{l} = \C e \oplus \C h \oplus \C f$ and we have  $$\xymatrix {
    \mathcal{O}(PSL_2(\C)) \ar@{^{(}->}[r]^{\iota} \ar@{=}[d]
    & \mathcal{O}_q(SL_2(\C)) \ar@{->>}[r]^{\pi} \ar@{=}[d] & \mathfrak{u}_{2,q}(\mathfrak{sl}_2)^{\circ} \ar@{=}[d] \\
    \mathcal{O}(PSL_2(\C)) \ar@{^{(}->}[r]^{\iota}  & \mathcal{O}_q(SL_2(\C)) \ar@{->>}[r]^{\pi} & \mathfrak{u}_{2,q}(\mathfrak{sl}_2)^{\circ}
    }$$

\end{enumerate}

For the second step, let $\sigma: \Gamma \rightarrow PSL_2(\C)$  be the inclusion of the subgroup $\Gamma$ of $PSL_2(\C)$ such that $\sigma(\Gamma) \subseteq \bar{L}$ and $\sigma^t : \mathcal{O}(\bar{L}) \rightarrow \mathcal{O}(\Gamma)$. Then, by Proposition \ref{prop2mu} the quotients of $\mathcal{O}_q(L)$ fit in an exact sequence and we have the following commutative diagram

$$\xymatrix {
    \mathcal{O}(PSL_2(\C)) \ar@{^{(}->}[r]^{\iota} \ar@{->>}[d]^{u_L} 
    & \mathcal{O}_q(SL_2(\C)) \ar@{->>}[r]^{\pi} \ar@{->>}[d]^{v_L} & \mathfrak{u}_{2,q}(\mathfrak{sl}_2)^{\circ} \ar@{->>}[d]^{w_L} \\
    \mathcal{O}(\bar{L}) \ar@{^{(}->}[r]^{\iota_L} \ar@{->>}[d]^{\sigma^t}  & \mathcal{O}_q(L) \ar@{->>}[r]^{\pi_L} \ar@{->>}[d]^{s} & \mathfrak{u}_{2,q}(\mathfrak{l})^{\circ} \ar@{=}[d] \\
    \Oc(\Gamma) \ar@{^{(}->}[r]^{\widehat{\iota}} & \mathcal{O}_q(L)/(\mathcal{J}) \ar@{->>}[r]^{\widehat{\pi}} & \mathfrak{u}_{2,q}(\mathfrak{l})^{\circ}
    }$$

where $\mathcal{J}=\text{Ker } \sigma^t$ and $\Oc(\Gamma) \cong \mathcal{O}(\bar{L})/\mathcal{J}$.


\begin{example} \label{zn2}

Recall the inclusion $\sigma: \Z_n \to PSL_2(\C)$ given in Example \ref{zn}. By the same computation $\ker \sigma^t= (X_{11}^{2n}-1, X_{22}^{2n}-1, X_{11}X_{22}-1,  X_{ij}X_{kl}  \text{ with } i \neq j \text{ or } k \neq l)$. 

Since $ad-qbc=1$, we have that $b(ad-qbc)=b$ and $s(b)=s(bad) - s(qbbc)= 0$. In the same way, $s(c)=0$. This implies that $b,c \in (\ker \sigma^{t})$.

Then, $\mathcal{O}_{q}(SL_2(\C))/(\ker \sigma^{t}) = \C \langle a | a^{2mn}=1 \rangle \cong \C ^{\Z_{2mn}} \cong \C \Z_{2mn} $ and we obtain the commutative diagram:

$$\xymatrix {
    \mathcal{O}(PSL_2(\C)) \ar@{^{(}->}[r]^{\iota} \ar@{->>}[d]^{u_T} 
    & \mathcal{O}_{q}(SL_2(\C)) \ar@{->>}[r]^{\pi} \ar@{->>}[d]^{v_T} &  \mathfrak{u}_{2,q}(\mathfrak{sl}_2)^{\circ} \ar@{->>}[d]^{w_T} \\
    \Oc(\bar{T}) \ar@{^{(}->}[r]^(.4){\iota_T}  \ar@{->>}[d]^{\sigma^t}  & \mathcal{O}_q(SL_2(\C))/ (b,c) \ar@{->>}[r]^(.3){\pi_T}  \ar@{->>}[d]^{s} &  \mathcal{O}_q(SL_2(\C))/ \langle a^{2m-1}, b,c, d^{2m-1}\rangle \cong \C \Z_{2m}  \ar@{=}[d] \\
    \C^{\Z_n} \ar@{^{(}->}[r]^{\widehat{\iota}}  & \C^{\Z_{2mn}} \ar@{->>}[r]^(.3){\widehat{\pi}} &  \mathcal{O}_q(SL_2(\C))/ \langle a^{2m-1}, b,c, d^{2m-1} \rangle \cong \C \Z_{2m}
    }$$

\end{example}
Finally, in case I, we take the quotient by the ideal generated by $b$ and $c$ and we have one more step. Let $N$ be a cyclic subgroup of $\Z_{\ell}$, suppose $N=(p) = \Z_t$ where $pt=\ell$, and define $\delta: N \to \widehat{\Gamma}\simeq (\chi)$ such that $\delta(p)= \chi^s$. One may take the quotient of $\mathcal{O}_q(L)/(\mathcal{J})$ by $J_{\delta}=(\bar{a}^p - \delta(p))$ and get another quantum subgroup. We obtain the commutative diagram

$$\xymatrix {
    \mathcal{O}(PSL_2(\C)) \ar@{^{(}->}[r]^{\iota} \ar@{->>}[d]^{u_L} 
    & \mathcal{O}_q(SL_2(\C)) \ar@{->>}[r]^{\pi} \ar@{->>}[d]^{v_L} & \mathfrak{u}_{2,q}(\mathfrak{sl}_2)^{\circ} \ar@{->>}[d]^{w_L} \\
    \mathcal{O}(\bar{L}) \ar@{^{(}->}[r]^{\iota_L} \ar@{->>}[d]^{\sigma^t}  & \mathcal{O}_q(L) \ar@{->>}[r]^{\pi_L} \ar@{->>}[d]^{s} & \mathfrak{u}_{2,q}(\mathfrak{l})^{\circ} \ar@{=}[d] \\
    \Oc(\Gamma) \ar@{^{(}->}[r]^{\widehat{\iota}} \ar@{=}[d] & \mathcal{O}_q(L)/(\mathcal{J}) \ar@{->>}[r]^{\widehat{\pi}} \ar@{->>}[d]^{t} & \mathfrak{u}_{2,q}(\mathfrak{l})^{\circ} \ar@{->>}[d]^{w} \\
    \Oc(\Gamma) \ar@{^{(}->}[r]^{\bar{\iota}} & A_{\mathcal{D}} \ar@{->>}[r]^{\bar{\pi}} & H
    }$$

where $A_{\mathcal{D}}$ is the quotient of  $\mathcal{O}_q(L)/(\mathcal{J})$ by $J_{\delta}$. Let us see this step with the following example.

\begin{example}

Continuing with Example \ref{zn2}, suppose that $m=n$ and take $N= (2)$ and $\delta(2)=\alpha$ where $\C ^{\Z_n} = \langle \alpha \rangle$. Then we take the quotient $\C^{\Z_{2mn}}/ \langle a^2=\alpha \rangle \cong \C^{\Z_{2n}} $ and we have the following diagram.

$$\xymatrix {
    \mathcal{O}(PSL_2(\C)) \ar@{^{(}->}[r]^{\iota} \ar@{->>}[d]^{u_T} 
    & \mathcal{O}_{q}(SL_2(\C)) \ar@{->>}[r]^{\pi} \ar@{->>}[d]^{v_T} &  \mathfrak{u}_{2,q}(\mathfrak{sl}_2)^{\circ} \ar@{->>}[d]^{w_T} \\
    \mathcal{O}(T) \ar@{^{(}->}[r]^(.4){\iota_T}  \ar@{->>}[d]^{\sigma^t}  & \mathcal{O}_q(SL_2(\C))/ (b,c) \ar@{->>}[r]^(.3){\pi_T}  \ar@{->>}[d]^{s} &  \mathcal{O}_q(SL_2(\C))/ \langle a^{2m-1}, b,c, d^{2m-1} \rangle \cong \C \Z_{2m} \ar@{=}[d] \\
    \C^{\Z_n} \ar@{^{(}->}[r]^{\widehat{\iota}} \ar@{=}[d]  & \C^{\Z_{2mn}} \ar@{->>}[r]^(.3){\widehat{\pi}} \ar@{->>}[d]^{t} &  \mathcal{O}_q(SL_2(\C))/ \langle a^{2m-1}, b,c, d^{2m-1} \rangle \cong \C \Z_{2m} \ar@{->>}[d]^{w} \\
    \C^{\Z_n} \ar@{^{(}->}[r]^{\bar{\iota}}  & \C^{\Z_{2n}} \ar@{->>}[r]^(.3){\bar{\pi}} &  \mathcal{O}_q(SL_2(\C))/ \langle a^{2}, b,c, d^{2} \rangle \cong \C \Z_2
    }$$
In particular, $\C^{\Z_{2n}}= \C \langle a,b,c,d, \alpha \ |\ b=c=0, ad=1, a^2=\alpha \rangle$.
\end{example}

\subsection{Determination of quantum subgroups.} \label{determinationeven}

The last step of the proof of Theorem \ref{teorema} is to see that this construction is comprehensive. We will prove this in the following lemma.

\begin{lemma} \label{lemadet}
Every quotient of Hopf algebras $\rho : \Oc_q(SL_2(\C)) \to A$ is determined by some even subgroup datum $\mathcal{D} = (I_{+}, I_{-}, N, \Gamma, \sigma, \delta)$. 
\end{lemma}

\begin{proof}
Given $\rho: \Oc_q(SL_2(\C)) \to A$, consider $K=\rho(\Oc(PSL_2(\C))$. Then $K$ is commutative and normal over $A$ so, by Proposition \ref{ff} the following sequence is exact

\begin{center}
$K \hookrightarrow A \twoheadrightarrow H \simeq A/AK^{+}$
\end{center}

As $K$ is a commutative Hopf algebra, there exists a group $\Gamma$ and a injective map $\sigma:\Gamma \hookrightarrow PSL_2(\C)$ such that $K \cong \Oc(\Gamma)$. Thus we have the commutative diagram

$$\xymatrix {
    \mathcal{O}(PSL_2(\C)) \ar@{^{(}->}[r]^{\iota}   \ar@{->>}[d]^{\sigma^t}
    & \mathcal{O}_q(SL_2(\C)) \ar@{->>}[r]^{\pi} \ar@{->>}[d]^{\rho} & \mathfrak{u}_{2,q}(\mathfrak{sl}_2)^{\circ}  \\
    \Oc(\Gamma) \ar@{^{(}->}[r]^{\bar{\iota}}  & A \ar@{->>}[r]^{\overline{\pi}} & H
    }$$

Since $\rho(\Oc(PSL_2(\C)))=K$ and $\rho(\Oc_q(SL_2(\C)))=A$ then
$$AK^{+}  = \rho(\Oc_q(SL_2(\C))) (\rho(\Oc(PSL_2(\C))))^{+} 
		= \rho(\Oc_q(SL_2(\C))(\Oc(PSL_2(\C)))^{+}) 
		= \rho(\Ker \pi).$$

That is, $\Ker(\bar{\pi})= \rho(\Ker{\pi})$. So $\Ker \pi \subseteq \Ker \bar{\pi}\rho$ and there exists an epimorphism $r: \mathfrak{u}_{2,q}(\mathfrak{sl}_2)^{\circ} \to H$ that makes the following diagram commutative

\begin{equation*} 
\xymatrix{
    \mathcal{O}(PSL_2(\C)) \ar@{^{(}->}[r]^{\iota}   \ar@{->>}[d]^{\sigma^t}
    & \mathcal{O}_q(SL_2(\C)) \ar@{->>}[r]^{\pi} \ar@{->>}[d]^{\rho} & \mathfrak{u}_{2,q}(\mathfrak{sl}_2)^{\circ} \ar@{->>}[d]^{r}  \\
    \Oc(\Gamma) \ar@{^{(}->}[r]^{\bar{\iota}}  & A \ar@{->>}[r]^{\overline{\pi}} & H
    }
\end{equation*}

Now we want to see that there exists some subgroup $L \subseteq SL_2(\C)$ such that the following diagram is commutative.

\begin{equation} \label{diag1}
\xymatrix{
    \Oc(PSL_2(\C)) \ar@{^{(}->}[r]^{\iota} \ar@{->>}[d]^{u_L} \ar@/_1.9pc/[dd]_{\sigma^t} 
    & \Oc_q(SL_2(\C)) \ar@/^1.9pc/[dd]^(0.65){\rho} \ar@{->>}[r]^{\pi} \ar@{->>}[d]^{v_L}  & 
    \mathfrak{u}_{2,q}(\mathfrak{sl}_2)^{\circ} \ar@{->>}[d]^{w_L} \ar@/^1.9pc/[dd]^{r} \\
    \Oc(\bar{L}) \ar@{^{(}->}[r]^{\iota_L} \ar@{-->}[d]^{u}  & \mathcal{O}_q(L) \ar@{->>}[r]^{\pi_L} \ar@{-->}[d]^{v} & \mathfrak{u}_{2,q}(\mathfrak{l})^{\circ} \ar@{-->}[d]^{w} \\
    \Oc(\Gamma) \ar@{^{(}->}[r]^{\bar{\iota}} & A \ar@{->>}[r]^{\bar{\pi}} & H
    }
\end{equation}

Since $H$ is a quotient of $\mathfrak{u}_{2,q}(\mathfrak{sl}_2)^{\circ}$, $H^{\circ}$ is a subalgebra of $\mathfrak{u}_{2,q}(\mathfrak{sl}_2)$. Hence $H^{\circ}$ is generated either by 

\begin{enumerate}
\item[(i)] $k,e,f$. In this case, $H^{\circ}= \mathfrak{u}_{2,q}(\mathfrak{sl}_2)$ and $L=SL_2(\C)$.
\item[(ii)] $k,e$. In this case, $H^{\circ}= \mathfrak{u}_{2,q}(\mathfrak{b}^{+})$ and $L=B^{+}$.
\item[(iii)] $k,f$. In this case, $H^{\circ}= \mathfrak{u}_{2,q}(\mathfrak{b}^{-})$ and $L=B^{-}$.
\item[(iv)] $k$. In this case, $H^{\circ}= \mathfrak{u}_{2,q}(\mathfrak{h})$ and $L=T$.
\item[(v)] $k^s$ for some $s$. In this case, $\Sigma=(k^s)\subseteq T$ and $\C \Sigma=H^{\circ} \subseteq \mathfrak{u}_{2,q}(\mathfrak{h})$. Let $w:\mathfrak{u}_{2,q}(\mathfrak{h})^{\circ} \to H=(\C \Sigma)^{\circ} \cong \C \Sigma$ and $\C N \cong \ker w$ with $N$ a subgroup of $\Z_{\ell}$. Then $N=(p) \cong \Z_t$ for some $t$ and $\bar{a}^p \in G(A_{L,\sigma})$.
\end{enumerate}

In all these cases, it follows by the construction we did before that $\Oc_q(L)$ fits into an exact sequence $\Oc(\bar{L}) \hookrightarrow \Oc_q(L) \twoheadrightarrow \mathfrak{u}_{2,q}(\mathfrak{l})^{\circ} $ and the diagram

\begin{equation*} 
\xymatrix{
    \mathcal{O}(PSL_2(\C)) \ar@{^{(}->}[r]^{\iota}   \ar@{->>}[d]^{u_L}
    & \mathcal{O}_q(SL_2(\C)) \ar@{->>}[r]^{\pi} \ar@{->>}[d]^{v_L} & \mathfrak{u}_{2,q}(\mathfrak{sl}_2)^{\circ} \ar@{->>}[d]^{w_L}  \\
    \Oc(\bar{L}) \ar@{^{(}->}[r]^{\iota_L}  & \Oc_q(L) \ar@{->>}[r]^{\pi_L} & \mathfrak{u}_{2,q}(\mathfrak{l})^{\circ}
    }
\end{equation*}

 commutes.

To prove that the diagram (\ref{diag1}) is commutative, we will first prove that $\ker v_L \subseteq \ker \rho$ so there exists $v : \Oc_q(L) \to A$ and then define $u= v \circ i_L$.

Recall that $\ker v_L$ is $\{0\}$ in case (i), $(c)$ in case (ii), $(b)$ in case (iii) and $(b,c)$ in cases (iv) and (v). We will only prove that $\rho(c)=0$. With the same argument one can prove $\rho(b)=0$. As $\bar{\pi}\rho=r\pi$ we have that $\bar{\pi} \rho(c)=0$. Suppose that $\rho(c) \neq 0$. Then $0 \neq \rho(c) \in \Ker \bar{\pi} = \rho(\Ker \pi)$ and $c \in \ker \pi + \ker \rho$. We can write $$c= \sum_{\alpha} \lambda_{\alpha} (a^{2m}-1)^{\alpha_1} (b^m)^{\alpha_2} (c^m)^{\alpha_3} (d^{2m}-1)^{\alpha_4} + \ker \rho \quad \text{for some } \lambda_{\alpha} = \lambda_{(\alpha_1,\alpha_2,\alpha_3,\alpha_4)}.$$

By \cite[(I.7.16.)]{Ken} $S=\{a^lb^mc^s | l,m,s \geq 0\} \cup \{ b^m c^s d^t | m,s \geq 0 \text{ and } t>0 \}$ is a PBW basis of $\Oc_q(SL_2(\C))$, so comparing degrees we have that $\lambda_{\alpha}=0$ for all $\alpha$. That is, $c \in \Ker \rho$ as we want to prove.

Now let $A_{L, \sigma}$ be the quotient of $\Oc_q(L)$ that is the pushout given in Proposition \ref{prop2mu}:

$$\xymatrix{
    \mathcal{O}(\bar{L}) \ar@{^{(}->}[r]^{\iota_L} \ar@{->>}[d]^{u} 
    & \mathcal{O}_q(L) \ar@{->>}[r]^{\pi_L} \ar@{->>}[d]^{s} & 
    \mathfrak{u}_{2,q}(\mathfrak{l})^{\circ} \ar@{=}[d] \\
    \Oc(\Gamma) \ar@{^{(}->}[r]^{\widehat{\iota}}  & A_{L, \sigma} \ar@{->>}[r]^{\widehat{\pi}}  & \mathfrak{u}_{2,q}(\mathfrak{l})^{\circ}
    }$$

where $\mathfrak{u}_{2,q}(\mathfrak{l})^{\circ}/(\widehat{\pi}(\bar{a}^p)-1) \simeq H$. By the commutativity of (\ref{diag1}), $v \iota_L= \bar{\iota}u$ so there exists a unique morphism $t: A_{L, \sigma} \to A$ that makes the diagram commutative.

\begin{equation*}
\begin{gathered}
\xymatrix{
\Oc(\bar{L}) \ar[r]^{\iota_L} \ar[d]_{u} & \Oc_q(L) \ar[d]^{s} \ar@/^1.2pc/[ddr]^{v} \\
\Oc(\Gamma) \ar[r]_{\widehat{\iota}} \ar@/_1.2pc/[drr]_{\bar{\iota}} 
  & A_{L,\sigma} \ar[dr]^{t} \ar@{}[ul]|{\ulcorner} \\
& & A
}
\end{gathered}
\end{equation*}

So $t \widehat{\iota}=\bar{\iota}$. Also, for any $x \in A_{L,\sigma}$ $ w \widehat{\pi}(x)= w \widehat{\pi} s(y) = w \pi_L (y)= \bar{\pi}v(y)= \bar{\pi} ts(y)=\bar{\pi}t(x)$ for some $y \in \Oc_q(L)$. Therefore, we obtain the following commutative diagram

$$\xymatrix {
    \mathcal{O}(PSL_2(\C)) \ar@{^{(}->}[r]^{\iota} \ar@{->>}[d]^{u_L} 
    & \mathcal{O}_q(SL_2(\C)) \ar@{->>}[r]^{\pi} \ar@{->>}[d]^{v_L} & \mathfrak{u}_{2,q}(\mathfrak{sl}_2)^{\circ} \ar@{->>}[d]^{w_L} \\
    \mathcal{O}(\bar{L}) \ar@{^{(}->}[r]^{\iota_L} \ar@{->>}[d]^{u}  & \mathcal{O}_q(L) \ar@{->>}[r]^{\pi_L} \ar@{->>}[d]^{s} & \mathfrak{u}_{2,q}(\mathfrak{l})^{\circ} \ar@{=}[d] \\
    \Oc(\Gamma) \ar@{^{(}->}[r]^{\widehat{\iota}} \ar@{=}[d] & A_{L,\sigma} \ar@{->>}[r]^{\widehat{\pi}} \ar@{->>}[d]^{t} & \mathfrak{u}_{2,q}(\mathfrak{l})^{\circ} \ar@{->>}[d]^{w} \\
    \Oc(\Gamma) \ar@{^{(}->}[r]^{\bar{\iota}} & A \ar@{->>}[r]^{\bar{\pi}} & H
    }$$

Note that when $H^{\circ}= \mathfrak{u}_{2,q}(\mathfrak{sl}_2), H^{\circ}= \mathfrak{u}_{2,q}(\mathfrak{b}^{+}), H^{\circ}= \mathfrak{u}_{2,q}(\mathfrak{b}^{-})$ or $H^{\circ}= \mathfrak{u}_{2,q}(\mathfrak{h})$, $w=id$. Thus, by Proposition \ref{propiso} it follows that $A \simeq A_{L, \sigma}$. \\

In the case $H^{\circ} \subseteq \mathfrak{u}_{2,q}(\mathfrak{h})$ and $\mathfrak{u}_{2,q}(\mathfrak{h})^{\circ} \simeq \C \Z_{\ell}$, $A$ is a quotient of $A_{L, \sigma}$ and $A_{L, \sigma}$ is a quotient of $\C[\bar{a},\bar{a}^{-1}]$. 
The last step of this proof is to see that there exists a group map $\delta: N \to \widehat{\Gamma}$ with $N= (p) \simeq \Z_t$ subgroup of $\Z_{\ell}$ such that $A \simeq A_{L, \sigma}/J_{\delta}$ where $J_{\delta}=  (\bar{a}^p - \delta(p))$.

\bigskip

Let $\C N \cong \ker w$ with $N$ a subgroup of $\Z_{\ell}$. Then $N=(p) \cong \Z_t$ for some $t$ and $\bar{a}^p \in G(A_{L,\sigma}), \bar{a}^p \neq 1$. Also $\bar{\pi} t (\bar{a}^p)= w \widehat{\pi}(\bar{a}^p)=1$ so $t (\bar{a}^p) \in \Ker(\bar{\pi}) = K^{+}A$. Moreover, since $t(\bar{a}^p)$ is a group like element, we have that $t(\bar{a}^p) \in K$.

As $t(\bar{a}^p) \in G(\Oc(\Gamma)) \simeq \widehat{\Gamma} = (\chi)$, then $t(\bar{a}^p)= \chi^s$ for some s. We define $\delta: N \to \widehat{\Gamma}$ as $\delta(p)= \chi^s$ and consider the ideal of $A_{L, \sigma}$ 

$$J_{\delta}=(\bar{a}^p - \delta(p)).$$

Since $J_{\delta} \subseteq \Ker t$, there exists a unique epimorphism $\eta: A_{L, \sigma}/J_{\delta} \twoheadrightarrow A$. Furthermore, by Proposition \ref{prop1mu}, we know that $A_{L, \sigma}/J_{\delta}$ fits into an exact sequence

$$ \Oc{(\Gamma})/\mathcal{J} \hookrightarrow A_{L, \sigma}/J_{\delta} \twoheadrightarrow  \mathfrak{u}_{2,q}(\mathfrak{l})^{\circ}/\widehat{\pi}(J_{\delta})$$

with $\mathcal{J}= J_{\delta} \cap \Oc(\Gamma)$. Since $\widehat{\pi}(\chi)=1$, $\widehat{\pi}(J_{\delta})$ is the ideal of $\mathfrak{u}_{2,q}(\mathfrak{l})^{\circ}$ given by $(\widehat{\pi}(\bar{a}^p)-1)$ and $\mathfrak{u}_{2,q}(\mathfrak{l})^{\circ}/\widehat{\pi}(J_{\delta}) \simeq H$. To finish the proof, let us see that $\mathcal{J}=0$. Then, by Proposition \ref{propiso}, $A \simeq A_{L,\sigma} /J_{\delta}$.

It is clear that $J_{\delta}=(\bar{a}^p - \delta(p))=(\bar{a}^p\delta(p)^{-1}-1)$.
Let $x \in \mathcal{J}= J_{\delta} \cap \Oc(\Gamma)$ and assume $x \neq 0$. Then,

$$x= \sum_{\alpha} \lambda_{\alpha} \bar{a}^{\alpha_1} (\bar{a}^p\delta(p)^{-1}-1)^{\alpha_2},$$
for some $\lambda_{\alpha}=\lambda_{\alpha_1, \alpha_2}$ with at least one $\alpha_2 \neq 0$. As $\widehat{\pi}(\bar{a}^p) \neq 1$ and 
$$\varepsilon(x)=\widehat{\pi}\widehat{\iota}(x)= \sum_{\alpha} \lambda_{\alpha} \widehat{\pi}(\bar{a}^{\alpha_1}) (\widehat{\pi}(\bar{a}^p)-1)^{\alpha_2}$$
comparing degrees we have that all $\alpha_2=0$, which is a contradiction. So $x=0$, that is, $\mathcal{J}= J_{\delta} \cap \Oc(\Gamma)=0$.

\end{proof}

\printbibliography



\end{document}